\documentclass[a4paper]{amsart}

\numberwithin{equation}{subsection}
\numberwithin{figure}{subsection}
\numberwithin{table}{subsection}

\usepackage{amssymb}
\usepackage{mathtools} 
\usepackage{amstext}
\usepackage{amsthm}

\usepackage{txfonts}
\usepackage{dsfont}
\usepackage{mathrsfs}

\usepackage{fancyhdr}
\usepackage{xcolor}
\usepackage[T1]{fontenc}

\usepackage{stmaryrd}
\usepackage{textcomp}

\usepackage{url}
\usepackage{hyperref}

\usepackage[cmtip,arrow,matrix,curve,tips,frame]{xy}
\xyoption{matrix}

\usepackage{tikz-cd}

\newcommand{\rpot}[1]{ (\hspace{-0,15em}( {#1} )\hspace{-0,15em})  }
\newcommand{\restr}[2]{{#1}\raise-.5ex\hbox{\ensuremath|}_{#2}}
\newcommand{\bigslant}[2]{{\raisebox{.2em}{$#1$}\hspace{-.1em}\left/ \hspace{-.1em}\raisebox{-.2em}{$#2$}\right.}}

\newcommand{\coh}[1]{\mathrm{H}^{#1}}

\newcommand{\cl}[1]{\mkern 1.5mu\overline{\mkern-1.5mu#1\mkern-1.5mu}\mkern 1.5mu}
\newcommand{\scl}[1]{{#1}^s}

\def \mono  {\hookrightarrow}

\def \epi   {\twoheadrightarrow}

\def \isom  {\stackrel{\sim}{\rightarrow}}

\def \bij   {\stackrel{1:1}{\rightarrow}}

\def \iv   {^{-1}}

\theoremstyle{plain}
\newtheorem{thm}[subsection]{Theorem}
\newtheorem*{thm*}{Theorem}
\newtheorem{thmsub}[equation]{Theorem}
\newtheorem{lem}[subsection]{Lemma}
\newtheorem*{lem*}{Lemma}
\newtheorem{lemsub}[equation]{Lemma}
\newtheorem{prop}[subsection]{Proposition}
\newtheorem*{prop*}{Proposition}
\newtheorem{propsub}[equation]{Proposition}
\newtheorem{cor}[subsection]{Corollary}
\newtheorem*{cor*}{Corollary}
\newtheorem{corsub}[equation]{Corollary}

\newtheorem*{claim*}{Claim}

\newtheorem*{conj*}{Conjecture}

\theoremstyle{definition}

\newtheorem{defn}[subsection]{Definition}
\newtheorem*{defn*}{Definition}
\newtheorem{defnsub}[equation]{Definition}

\newtheorem*{notn*}{Notation}

\theoremstyle{remark}
\newtheorem{rmk}[subsection]{Remark}
\newtheorem{rmksub}[equation]{Remark}
\newtheorem*{rmk*}{Remark}
\newtheorem{exa}[subsection]{Example}
\newtheorem*{exa*}{Example}
\newtheorem{exasub}[equation]{Example}

\DeclareMathOperator{\Ad}{Ad}

\DeclareMathOperator{\Aut}{Aut}
\DeclareMathOperator{\Cent}{Cent}

\DeclareMathOperator{\diag}{diag}
\DeclareMathOperator{\Div}{Div}
\DeclareMathOperator{\End}{End}

\DeclareMathOperator{\id}{id}

\DeclareMathOperator{\Int}{Int}

\DeclareMathOperator{\Gal}{Gal}
\DeclareMathOperator{\GL}{GL}

\DeclareMathOperator{\Hom}{Hom}

\DeclareMathOperator{\Nm}{Nm}
\DeclareMathOperator{\Pic}{Pic}

\DeclareMathOperator{\Prin}{PDiv}

\DeclareMathOperator{\SL}{SL}

\DeclareMathOperator{\Spec}{Spec}

\DeclareMathOperator{\SU}{SU}

\DeclareMathOperator{\Res}{Res}

\DeclareMathOperator{\val}{val}

\def \Jac {\Pic^0}

\def \ad    {\mathrm{ad}}
\def \der   {\mathrm{der}}

\newcounter{listnum}

\newcounter{asslistcounter}

\newcounter{subenvcounter}
\newenvironment{subenv}{%
 \begin{list}
  {\em (\arabic{subenvcounter})}
  {\setlength{\leftmargin}{20pt}
   \setlength{\rightmargin}{0pt}
   \setlength{\itemindent}{0pt}
   \setlength{\labelsep}{5pt}
   \setlength{\labelwidth}{13pt}
   \setlength{\listparindent}{\parindent}
   \setlength{\parsep}{0pt}
   \setlength{\itemsep}{0pt}
   \setlength{\topsep}{-\parskip}
   \usecounter{subenvcounter}}}
  {\end{list}}

\def \Hrm {{\mathrm{H}}}

\def \Nrm {{\mathrm{N}}}

\def \Zrm {{\mathrm{Z}}}

\def \bsf {{\mathsf{b}}}

\def \gsf {{\mathsf{g}}}

\def \jsf {{\mathsf{j}}}

\def \zsf {{\mathsf{z}}}
\def \Asf {{\mathsf{A}}}
\def \Bsf {{\mathsf{B}}}

\def \Dsf {{\mathsf{D}}}
\def \Esf {{\mathsf{E}}}
\def \Fsf {{\mathsf{F}}}
\def \Gsf {{\mathsf{G}}}
\def \Hsf {{\mathsf{H}}}

\def \Jsf {{\mathsf{J}}}

\def \Msf {{\mathsf{M}}}

\def \Psf {{\mathsf{P}}}

\def \Rsf {{\mathsf{R}}}
\def \Ssf {{\mathsf{S}}}
\def \Tsf {{\mathsf{T}}}
\def \Usf {{\mathsf{U}}}

\def \Zsf {{\mathsf{Z}}}

\def \bbf {{\mathbf{b}}}

\def \hbar {{\overline{h}}}

\def \tbar {{\overline{t}}}

\def \Gscr {{\mathscr{G}}}

\def \Vscr {{\mathscr{V}}}

\def \AA {{\mathbb{A}}}

\def \DD {{\mathbb{D}}}

\def \FF {{\mathbb{F}}}
\def \GG {{\mathbb{G}}}

\def \NN {{\mathbb{N}}}

\def \QQ {{\mathbb{Q}}}

\def \ZZ {{\mathbb{Z}}}

\def \FFbar {{\overline{\mathbb{F}}}}

\def \Ebreve {{\breve{E}}}
\def \Fbreve {{\breve{F}}}

\usepackage[british]{babel}

\def \unif {\varpi}
\def \A {\mathrm{A}}
\def \B {\mathrm{B}}

\newcommand{\Vals}[1]{\Div({#1})_0}

\DeclareMathOperator{\loc}{loc}
\DeclareMathOperator{\cores}{cor}

\author[P.~Hamacher, W.~Kim]{Paul Hamacher and Wansu Kim}
\address{Paul Hamacher\\
Technische Universit\"at M\"unchen\\
Zentrum Mathematik - M 11\
Boltzmannstra{\ss}e 3\\
85748 Garching\\
Deutschland\\ \newline
Wansu Kim\\%
Department of Mathematical Sciences\\%
KAIST\\%
291 Daehak-ro, Yuseong-gu\\%
Daejeon, 34141\\%
South Korea}
\email{hamacher@ma.tum.de, wansu.math@kaist.ac.kr}



\title[On $\Gsf$-isoshtukas over function fields]{On $\Gsf$-isoshtukas over function fields}

\begin{document}

 \begin{abstract}
  In this paper we classify isogeny classes of global $\Gsf$-shtukas over a smooth projective curve $C/\FF_q$ (or equivalently $\sigma$-conjugacy classes in $\Gsf(\Fsf \otimes_{\FF_q} \cl{\FF_q})$ where $\Fsf$ is the field of rational functions of $C$) by two invariants $\bar\kappa,\bar\nu$ extending previous works of Kottwitz. This result can be applied to study points of moduli spaces of $\Gsf$-shtukas and thus is helpful to calculate their cohomology.
 \end{abstract}

 \maketitle

 \section{Introduction}
 Let $\FF_q$ be the finite field with $q$ elements and let $C$ be a curve (= geometrically integral smooth projective scheme of dimension $1$) over $\FF_q$. We denote by $\Fsf$ the field of rational functions on $C$.  We denote by $k$ an algebraic closure of $\FF_q$ and let $\breve\Fsf = \Fsf \otimes_{\FF_q} k$. The Frobenius automorphism on $k$ induces an automorphism $\sigma$ on $\breve\Fsf$.

 Let $\Gsf$ be a (connected) reductive group over $\Fsf$. A \emph{$\Gsf$-isoshtuka} over $k$ is a $\Gsf$-torsor $\Vscr$ over $\breve\Fsf$ together with an isomorphism $\phi\colon \sigma^\ast\Vscr \isom \Vscr$. 
 Such a notion naturally arises as the `generic fibre' of a $\Gscr$-shtuka over $k$ where $\Gscr$ is a smooth affine group scheme over $C$ with generic fibre $\Gsf$. The generic fibre plays the role of $\Gscr$-shtuka up to isogeny, hence the terminology `$\Gsf$-isoshtuka'.
 
 Since $\breve\Fsf$ has cohomological dimension one by Tsen's theorem, any $\Gsf$-torsor $\Vscr$ is trivial by \cite[\S~8.6]{BorelSpringer:RationalityPropertiesII}. Choosing a trivialisation $\Vscr \cong \Gsf_{\breve\Fsf}$, $\phi$ gets identified with the automorphism $\bsf \circ \sigma$ for some $\bsf \in \Gsf(\breve\Fsf)$. Every other trivialisation of $\Vscr$ can be obtained by postcomposing the above isomorphism with an element $\gsf \in \Gsf(\breve\Fsf)$, thus replacing $\bsf$ by $\gsf \bsf \sigma(\gsf^{-1})$. Hence this construction yields a natural bijection between the isomorphism classes of $\Gsf$-isoshtukas over $k$ and the set of $\sigma$-conjugacy classes in $\Gsf(\breve\Fsf)$. 

 This paper studies the pointed set $\B(\Fsf,\Gsf)$ of $\sigma$-conjugacy classes in $\Gsf(\breve\Fsf)$. Following the strategy of
Kottwitz' work \cite{Kottwitz:Gisoc1},\cite{Kottwitz:Gisoc2} on $\sigma$-conjugacy classes over $p$-adic fields and his construction of $B(F,G)$ for local and global fields in terms of Galois gerbs in \cite{Kottwitz:Gisoc3}, we
describe its elements via two invariants $\nu_\Gsf$ and $\kappa_\Gsf$ on $\B(\Fsf,\Gsf)$.
 
 Let us give more details on $\nu_\Gsf$ and $\kappa_\Gsf$.
 For any finite field extension $\Esf/\Fsf$, we denote by $\Div(\Esf)$ the free abelian group  generated by the set of places in $\Esf$ and let
 \[
  \Vals{\Esf} = \left\{ \sum n_y \cdot y \in \Div(\Esf) \mid \sum n_y = 0 \right\}.\footnote{Note that this definition is different from the subgroup of degree zero divisors $\Div^0(\Esf) = \left\{ \sum n_y \cdot y \in \Div(\Esf) \mid \sum n_y \cdot \deg(y) = 0 \right\}$. This is due to the fact that we are actually considering the Galois coinvariants of $\Div^0(\breve\Fsf)$ where every place has degree one,  see \S~\ref{ssect-Div0} for details.}
 \]
 For every finite extension $\Esf'/\Esf$, we obtain a homomorphism $\Vals{\Esf} \to \Vals{\Esf'}, x \mapsto \sum_{x'|x} [\Esf'_{x'}: \Esf_{x}] x'$. We denote by  $\Vals{\scl{\Fsf}} = \varinjlim \Vals{\Esf}$, where $\Esf$ runs through all finite separable extensions of $\Fsf$ and let $\DD_{\Fsf}$ be the $\Fsf$-protorus with character group $\Vals{\scl{\Fsf}}$.  In sections~\ref{sect-tori} and \ref{sect-redgps} we construct  invariants
 \begin{align*}
  \bar\kappa_\Gsf\colon &\B(\Fsf,\Gsf) \to (\pi_1(\Gsf) \otimes \Vals{\scl{\Fsf}})_{\Gal(\scl{\Fsf}/\Fsf)} \\
  \bar\nu_\Gsf\colon &\B(\Fsf,\Gsf) \to \left(\bigslant{\Hom_{\breve\Fsf}(\DD_\Fsf,\Gsf)}{\Gsf(\breve\Fsf)} \right)^{\sigma},
 \end{align*}
 which we call the Kottwitz map and the Newton map, respectively. 
 
 The above maps can be localised to obtain the Kottwitz point and the Newton point over a local field of $\Fsf$. More precisely, let $x$ be a closed point of $C$ and denote by $F_x$ the completion of $\Fsf$ at $x$. We fix an embedding of separable closures $\scl{\Fsf} \mono F_x^s$. 
 
 \begin{prop} \label{main-prop}
  The map $\bsf \mapsto \bsf \cdot \sigma(\bsf) \cdots \sigma^{\deg(x) -1}(\bsf)$ induces a map $N_x\colon \Bsf(\Fsf,\Gsf) \to \B(F_x,\Gsf)$. Moreover, there exist morphisms $\iota_x\colon \DD \mono \DD_\Fsf$, $\loc_x \colon A(\Fsf,\Gsf) \to \pi_1(\Gsf)_{\Gal(F_x^s/F_x)}$ (see sections~\ref{sect-tori} and \ref{sect-redgps} for their definition) such that for every $\bbf \in \B(\Fsf,\Gsf)$
  \begin{align*}
   \bar\kappa_{\Gsf_{F_x}}(N_x(\bsf)) &= \loc_x(\bar\kappa_\Gsf(\bbf)) \\
   \bar\nu_{\Gsf_{F_x}}(N_x(\bsf)) &= \bar\nu_\Gsf(\bbf) \circ \iota_x,
  \end{align*}
  where the terms on the left hand side are the Newton and Kottwitz point for the local field $F_x$ as defined in \cite{Kottwitz:Gisoc1}.
 \end{prop} 
 
 (see Corollary~\ref{cor-Kottwitz-pt-properties} and Lemma~\ref{lem-Newton-pt-properties} for details).
 
  Interestingly, the set $\B(\Fsf,\Gsf)$ shares a lot properties with its analogue over local fields. We show that $\bar\nu_\Gsf(\bbf)$ is trivial if and only if $\bbf$ lies in the image of $\coh{1}(\Fsf,\Gsf) \mono \B(\Fsf,\Gsf)$. More generally, we call $\bbf \in \B(\Fsf,\Gsf)$ basic, if $\bar\nu_\Gsf(\bbf)$ factors through the center of $\Gsf$. In section~\ref{sect-basic} we give the following classification of basic $\sigma$-conjugacy classes. In particular, this gives complete description of $\B(\Fsf,\Gsf)$ when $\Gsf$ is a torus.
 
 \begin{thm} \label{main-thm 1}
  The Kottwitz map induces an isomorphism $\B(\Fsf,\Gsf)_b \isom (\pi_1(\Gsf) \otimes \Vals{\scl{\Fsf}})_{\Gal(\scl{\Fsf}/\Fsf)}$.
 \end{thm}
 
 To obtain a description of the whole $\B(\Fsf,\Gsf)$ by its invariants, we proceed as follows. By giving a combinatorial description how $\B(\Fsf,\Gsf)$ behaves under ad-isomorphisms, we may reduce to the case that $\Gsf$ is of adjoint type. In particular,the quasi-split inner form $\Gsf^\ast$ of $\Gsf$ is an (extended) pure inner form. From that we deduce that  $\B(\Fsf,\Gsf) \cong \B(\Fsf,\Gsf^\ast)$.   Thus it suffices to describe $\B(\Fsf,\Gsf)$ for quasi-split $\Gsf$. In this case, we can reduce to the theorem above since every $\sigma$-conjugacy class in $\Gsf$ is induced by a $\sigma$-conjugacy class of an $\Fsf$-torus in $\Gsf$. More precisely, we get the following result. 

 \begin{thm} \label{main-thm 2}
  Let $\Gsf$ be a reductive group.
  \begin{subenv}
   \item Every $\bbf \in \B(\Fsf,\Gsf)$ is uniquely determined by its invariants $\bar\kappa_\Gsf(\bbf)$ and $\bar\nu_\Gsf(\bbf)$.
   \item If $\Gsf$ is quasi-split, the canonical map
   \[
    \bigcup_{\Tsf \subset \Gsf \atop \textnormal{max. } \Fsf- \textnormal{torus}} \B(\Fsf,\Tsf) \to \B(\Fsf,\Gsf)
   \]
   is surjective.
  \end{subenv}
 \end{thm}

 The second part of statement can be seen as an analogue for the moduli space of global $\Gscr$-shtukas to the statement that every isogeny class in the special fibre of a Shimura variety contains a point that can be lifted to a CM-point. A proof of the latter statement for Shimura varieties of PEL-type was first sketched in a letter of Langlands to Rapoport, and was proven in differing generality by Milne \cite{Milne:ModpPointsGoodRed}, Zink \cite{Zink:Isogenieklassen}, Kottwitz \cite{Kottwitz:PtShimuraVarFinFields}, Kisin \cite{Kisin:LanglandsRapoport} and most recently by Zhou \cite{Zhou:Modp} for Shimura varieties of Hodge type with paraholic level structure at $p$ given that certain group theoretic conditions are satisfied.

 The classification of $\B(\Fsf,\Gsf)$ is a generalisation of Drinfeld's classification of $\varphi$-spaces (\cite{Drinfeld:PeterssonsConjecture}, see also \cite{Laumon-Rapoport-Stuhler}). A $\varphi$-space over $k$ is an $\breve\Fsf$-vector space $V$ together with a $\sigma$-semilinear bijection $\varphi\colon \sigma^\ast V \isom V$ (i.e.\ a $\GL_{\dim V}$-isoshtuka). Drinfeld proved that the category of $\varphi$-spaces is semi-simple and that its simple objects are parametrised by pairs $(\tilde\Fsf,\tilde\Pi)$, where $\tilde\Fsf$ is a separable finite field extension and $\tilde\Pi \in \tilde\Fsf^\times \otimes \QQ$ does not belong to $\Fsf'^\times \otimes \QQ$ for any intermediate field $\Fsf \subset \Fsf' \subsetneq \tilde\Fsf$. In above terms this is stated as follows. Let $(V,\varphi)$ have simple factors which correspond to pairs $(\tilde\Fsf_i,\tilde\Pi_i)$. We note that we have an isomorphism \begin{align*} 
 {\rm div}\colon (\tilde\Fsf_i^\times \otimes \QQ) &\isom \Vals{\tilde\Fsf_i} \otimes \QQ \\
  \tilde\Pi_i &\mapsto \sum  x(\tilde\Pi_i) \cdot \deg(x) \cdot x,
 \end{align*}
 since the kernel and cokernel of the ``usual'' divisor map $\tilde\Fsf_i^\times \to \Div^0(\tilde\Fsf_i)$ are unit roots and points in the Jacobian, respectively, and in particular torsion.
 Denote by $d_i$ the common denominator of $x(\tilde\Pi_i) \cdot \deg(x)$. By choosing a representative of the Newton point in $\GL_{\dim V}$ which maps to the diagonal torus $\Tsf$ , we can describe it as an element of 
 \[
  \Hom_{\scl{\Fsf}}(\DD_F,\Tsf) = \Hom_{\ZZ} (X^\ast(\Tsf),X^\ast(\DD_\Fsf)) = \Vals{\scl\Fsf}^{\dim V}
 \]
 unique up to action of the Weyl group, i.e.\ up to permutation. Then the the $\sigma$-conjugacy class $\bbf\in \B(\Fsf,\GL_{\dim V})$ defined by $(V,\varphi)$ is uniquely determined by its Newton point $(\dotsc,\underbrace{{\rm div}(\tau_1(\tilde\Pi_i)),\dotsc,{\rm div}(\tau_{n_i}(\tilde\Pi_i))}_{d_i \textnormal{ times repeated}},\dotsc)$, where $\tau_1,\dotsc,\tau_{n_i}\colon \tilde\Fsf_i \mono \scl{\Fsf}$ denote the $\Fsf$-linear embeddings of $\tilde\Fsf_i$.
 
 The essential image of the Newton map is determined by Theorem~\ref{main-thm 2}~(2); any maximal torus of $\GL_{\dim V}$ is of the form $\prod \Res_{\Esf_i/\Fsf} \GG_m$ where $\Esf_i/\Fsf$ are separable field extensions of cumulative degree $\dim V$. By Theorem~1.1 and the construction of the Newton point given in section~\ref{sect-tori}, they contribute Newton points of the form
  \[
   (\dotsc,\tau_1(D_i),\dotsc,\tau_{m_i}(D_i),\dotsc),
  \]
  where $D_i \in \Vals{\Esf_i}$ and $\tau_1,\dotsc,\tau_{m_i}\colon \Esf \to \sc{\Fsf}$ denote the $\Fsf$-linear embeddings. 

These results about $\B (\Fsf,\Gsf)$ provide useful tools to study points in the special fibre of moduli space of $\Gscr$-shtukas. There is an extremely long list of previous results on `point-counting' on the moduli space of $\Gscr$-shtukas, most of which partition points by isogeny classes. For the most recent results applicable beyond inner forms of $\GL_n$, see \cite{ArastehRad-Hartl:LR, NgoNgoDac:PointCountingRegEll, NgoDac:PointCountingEllipticPart, NgoDac:PointCounting}.

Another natural question to ask is whether the pointed set of $\sigma$-conjugacy classes $\B(\Fsf,\Gsf)$ and the pointed set of Galois gerbs constructed by Kottwitz in \cite{Kottwitz:Gisoc3} are the same, i.e.\ whether there exists a canonical isomorphism of functors.  This question is proven to have a positive answer in an upcoming work of Iakovenko \cite{Iakovenko}.

 \subsection*{Notation and conventions} For any finite extension $\Esf/\Fsf$ we denote by $\breve\Esf = \Esf\cdot \breve\Fsf$ the maximal unramified extension. Let $\sigma_\Esf\in \Gal(\breve\Esf/\Esf)$ denote the lift of the Frobenius of $k$ over the field of constants $k_\Esf$ of $\Esf$. Let $C_{\Esf}$ denote the smooth projective curve associated to $\Esf$. We denote by $\Div(\Esf)$ the group of divisors on $C_{\Esf}$ and by $\Prin(\Esf) \subset \Div(\Esf)$ the subgroup of principal divisors.

 For $y \in |C_{\Esf}|$ we denote by $E_y$ the $y$-adic completion of $\Esf$, by $\Ebreve_y$ the completion of its maximal unramified extension and by $\sigma_y\colon \Ebreve_y \to \Ebreve_y$ the Frobenius morphism over $E_y$. 
 For every field $F$, we denote by $\scl{F}$ the separable closure of $F$. For every $x \in |C|$ we fix an embedding $\scl{\Fsf} \mono \scl{F_x}$ and denote by $y_x$ the corresponding continuation of $x$ to $\scl\Fsf$.

\subsection*{Acknowledgements:} We are grateful to Sergei Iakovenko for helpful comments and to Peter Scholze for pointing out an error in the previous version of this manuscript. We thank Brian Conrad and Paul Ziegler for helpful discussions, and the anonymous referee(s) for the careful reading and comments. The first named author was partially supported by ERC Consolidator Grant 770936: NewtonStrat. The second named author was supported by the National Research Foundation of Korea(NRF) grant funded by the Korea government(MSIT) (No.~2020R1C1C1A0100945311).

 \section{Preliminaries} 
 In this section we allow $\Gsf$ to be a  linear algebraic group over the function field $\Fsf$. In the following sections, we will assume that $\Gsf$ is reductive.

 \subsection{} \label{ss prelim}
 By evaluating a crossed homomorphism $f \in \Zrm^1(\sigma^\ZZ,\Gsf(\breve\Fsf))$ at $\sigma$ we obtain a natural isomorphism $\coh{1}(\sigma^\ZZ, \Gsf(\breve\Fsf)) \cong \B(\Fsf,\Gsf)$. In particular, the canonical morphism $\sigma^\ZZ \mono \Gal(\breve\Fsf/\Fsf)$ induces an embedding
 
 \begin{equation} \label{H1-B}
  \coh{1}(\breve\Fsf/\Fsf,\Gsf) \mono \B(\Fsf,\Gsf).
 \end{equation}
 Moreover, we obtain restriction morphisms
 \begin{align}
  \label{B-field-extn} \B(\Fsf,\Gsf) &\to \B(\Fsf',\Gsf) \\
  \label{B-localisation} \B(\Fsf,\Gsf) &\to \B(F_x,\Gsf) 
 \end{align}
 for any finite field extension $\Fsf'/\Fsf$ and any place $x \in |C|$. Explicitly, these morphism are given as follows. For $\gsf\in \Gsf(\breve\Fsf)$ and $d \in \NN$ let 
 \[
 \Nrm^{(d)}(\gsf) \coloneqq \gsf \cdot \sigma(\gsf) \dotsm \sigma^{d-1}(\gsf).
 \]
 By a similar argument as above, we have \[\coh{1}(\breve\Fsf/\Fsf,\Gsf) \cong \coh{1}(\sigma^{\hat\ZZ},\Gsf(\breve\Fsf)) = \varinjlim \coh{1}(\sigma^{\ZZ/d\ZZ},\Gsf(\Fsf \otimes \FF_{q^d})),\] thus (\ref{H1-B}) identifies $\coh{1}(\breve\Fsf/\Fsf,\Gsf)$ with the $\sigma$-conjugacy classes $\bbf \in \B(\Fsf,\Gsf)$ such that for some (or equivalently every) $\bsf \in \bbf$ we have $\Nrm^{(d)}(\bsf) = 1$ for some $d\in\NN$ divisible enough. The morphisms (\ref{B-field-extn}) and (\ref{B-localisation}) are induced by $\gsf \mapsto \Nrm^{(d)}(\gsf)$, where $d$ denotes the degree of the field of constants in $\Fsf'$ over $\FF_q$, respectively the degree of $x$.
 
 If $\Gsf$ is abelian, we obtain in addition the corestriction morphism
 \begin{equation}
  \label{B-cores} \B(\Fsf',\Gsf) \to \B(\Fsf,\Gsf)
 \end{equation}
 for any finite field extension $\Fsf'/\Fsf$. Explicitly, this map is given by $\gsf \mapsto \prod\tau(\gsf)$ where the product is taken over all $\breve\Fsf$-linear embeddings $\breve\Fsf' \mono \scl\Fsf$.
 
 \subsection{} \label{ss-exaxt-seq}
  Assume for the moment that $\coh{1}(\breve\Fsf,\Gsf) = 1$. This is for example the case when $\Gsf$ is reductive (\cite[\S~8.6]{BorelSpringer:RationalityPropertiesII}), split unipotent or any extension thereof. As the Weil group $W_\Fsf$ fits inside a short exact sequence
 \begin{center}
  \begin{tikzcd}
   1 \ar{r} & \Gal(\scl{\Fsf}/\breve\Fsf) \ar{r} & W_\Fsf \ar{r} & \sigma^\ZZ \ar{r} & 1,
  \end{tikzcd}
 \end{center}
 we obtain the inflation-restriction exact sequence
 \begin{center}
  \begin{tikzcd}
   1 \ar{r} & \coh{1}(\sigma^\ZZ,\Gsf(\breve\Fsf)) \ar{r} & \coh{1}(W_\Fsf,\Gsf(\scl\Fsf)) \ar{r} & \coh{1}(\breve\Fsf,\Gsf).
  \end{tikzcd}
 \end{center}
 Since $\coh{1}(\breve\Fsf,\Gsf) = 1$ by assumption, we thus obtain a natural isomorphism $\coh{1}(W_\Fsf,\Gsf(\scl\Fsf)) \cong \B(\Fsf,\Gsf)$. Similarly, we have $\coh{1}(\Fsf,\Gsf) \cong \coh{1}(\breve\Fsf/\Fsf,\Gsf)$. 
 
 Let $1 \to \Gsf_1 \to \Gsf_2 \to \Gsf_3 \to 1$ be an exact sequence of linear algebraic groups with $\coh{1}(\breve\Fsf,\Gsf_1) = 1$. Thus the sequence
 \begin{center}
  \begin{tikzcd}
   1 \ar{r} & \Gsf_1(\breve\Fsf) \ar{r} & \Gsf_2(\breve\Fsf) \ar{r} & \Gsf_3(\breve\Fsf) \ar{r} & 1
  \end{tikzcd}
 \end{center}
 is also exact. Taking the long exact cohomology sequence for $\sigma^\ZZ$, we obtain  
 \begin{center}
  \begin{tikzcd}
   1 \ar{r} 
   & \Gsf_1(\Fsf) \ar{r} 
   & \Gsf_2(\Fsf) \ar{r} \ar[phantom, ""{coordinate, name=Z}]{d}
   & \Gsf_3(\Fsf) \ar[rounded corners, to path={--([xshift=2ex]\tikztostart.east)|- (Z)[near end]\tikztonodes-| ([xshift=-2ex]\tikztotarget.west)-- (\tikztotarget)}]{dll} & \\
   & \B(\Fsf,\Gsf_1) \ar{r} 
   & \B(\Fsf,\Gsf_2) \ar{r} 
   & \B(\Fsf,\Gsf_3) \ar{r} 
   & 1,
  \end{tikzcd}
 \end{center}
 where the surjectivity of $\B(\Fsf,\Gsf_2) \to \B(\Fsf,\Gsf_3)$ is an immediate consequence of the surjectivity of $\Gsf_2(\breve\Fsf) \to \Gsf_3(\breve\Fsf)$.
 
  \subsection{}
  Let $\Fsf'/\Fsf$ be a finite field extension and $\Gsf'$ be a linear algebraic group over $\Fsf'$. We denote by $\Res_{\Fsf'/\Fsf} \Gsf'$ the $\Fsf$-group obtained from $\Gsf'$ by restriction of scalars and let $\sigma^d \coloneqq \sigma'$ denote the Frobenius over $\Fsf'$. Then $(\Res_{\Fsf'/\Fsf} \Gsf')(\breve\Fsf)$ is the induced $\sigma^\ZZ$-group from the $\sigma'^\ZZ$-group $\Gsf'(\breve\Fsf')$. Thus Shapiro's lemma (see e.g.~\cite[Prop.~8]{Stix:NonAbShapiro}) tells us that we have an isomorphism
  \begin{equation} \label{B-shapiro}
   \B(\Fsf,\Res_{\Fsf'/\Fsf} \Gsf') \isom \B(\Fsf',\Gsf')
  \end{equation}
 given by $\gsf \mapsto \Nrm^{(d)}(m(\gsf))$ where $m$ denotes the multiplication map $(\Res_{\Fsf'/\Fsf}\Gsf')(\breve\Fsf)=\Gsf'(\Fsf' \otimes_\Fsf \breve\Fsf) \to \Gsf'(\breve\Fsf')$. Note that we may interpret (\ref{B-field-extn}) and (\ref{B-localisation}) as composition of
 \begin{align*}
  \B(\Fsf,\Gsf) &\to \B(\Fsf,\Res_{\Fsf'/\Fsf}\Gsf) \cong \B(\Fsf',\Gsf) \\
  \B(\Fsf,\Gsf) &\to \B(\Fsf,\Res_{F_x/\FF_q\rpot{\unif_x}} \Gsf) \cong \B(\Fsf_x,\Gsf)
 \end{align*} 
 where the first morphism is induced by the canonical embedding $\Gsf(\breve\Fsf) \mono \Gsf(\breve\Fsf \otimes_{\Fsf} \Fsf')$ and $\Gsf(\breve\Fsf) \mono \Gsf(\breve\Fsf \hat\otimes_{\FF_q\rpot{\unif_x}} F_x)$, respectively and the isomorphism is given by (\ref{B-shapiro}).
  
  \begin{exa} \label{exa-Gm}
  Assume that $\Gsf = \GG_m$. By taking coinvariants of the short exact sequence
  \begin{center}
   \begin{tikzcd}
   1 \ar{r} & \FFbar_q^\times \ar{r}& \breve\Fsf^\times \ar{r}{\rm div}& \Prin(\breve\Fsf) \ar{r} & 1,
   \end{tikzcd}
  \end{center}
  we obtain
  $\B(\Fsf,\GG_m) \cong \Prin(\breve\Fsf)_{\sigma}$ since every element of $\FFbar_q^\times$ can be written as $\sigma(x) \cdot x\iv = x^{q-1}$ and hence $(\FFbar_q^\times)_\sigma = 1$. 
 By Shapiro's lemma we obtain for any finite separable extension $\Esf/\Fsf$ that $\B(\Fsf,\Res_{\Esf/\Fsf}\GG_m) \cong \Prin(\breve\Esf)_{\sigma_\Esf}$.
 \end{exa}   
  
 \subsection{} \label{ssect-Div0}
 To get an explicit description of $\Prin(\breve\Esf)_{\sigma_\Esf}$, consider the exact sequence
 \begin{center}
  \begin{tikzcd}
   \Hrm_1(\sigma^\ZZ,\Jac(C_\Esf)(\bar\FF_q)) \ar{r} &\Prin(\breve\Esf)_{\sigma_\Esf} \ar{r} & \Div^0(\breve\Esf)_{\sigma_\Esf} \ar{r} & \Jac(C_\Esf)(\bar\FF_q)_{\sigma_\Esf},
  \end{tikzcd} 
 \end{center}
 where we identified $\Jac(C_\Esf)(\bar\FF_q) = \Jac(C_{\breve\Esf})(\bar\FF_q)$ (see e.g.\ \cite[Prop.~8.1.4]{neronmodels}). Since the the Lang isogeny is surjective, the right most term is trivial and thus $\Prin(\breve\Esf)_{\sigma_\Esf} \to \Div^0(\breve\Esf)_{\sigma_\Esf}$ is surjective. As $\Pic^0(C_\Esf)(\FFbar_q)$ is torsion, so is the left most term. But $\Prin(\breve\Esf)_{\sigma_\Esf,{\rm tors}} \cong \B(\Esf,\GG_m)_{\rm tors}$, i.e.\ a torsion element corresponds to a class $[x] \in (\breve\Esf^\times)_{\sigma_\Esf}$ such that there exist $y \in \breve\Esf, n\in \NN$ such that $x^n = y\cdot \sigma_E(y)\iv$. In particular, if we fix a finite subextension $\Esf' \subset \breve\Esf$ containing $x$ and $y$, we obtain $N_{\Esf'/\Esf}(x)^n =  N_{\Esf'/\Esf}(y\cdot \sigma_E(y)\iv) = 1$. By further enlarging $\Esf'$ by a degree $n$ subextension of $\breve\Esf$, we obtain $N_{\Esf'/\Esf}(x) = 1$ and thus is already of the form $y'\cdot \sigma(y')\iv$ for some $y' \in \Esf'$ by Hilbert 90. In other words, $[x] = [1]$ and hence  we have $\Prin(\breve\Esf)_{\sigma_\Esf,{\rm tors}} = 1$. Altogether, we have shown that $\Prin(\breve\Esf)_{\sigma_\Esf} \cong \Div^0(\breve\Esf)_{\sigma_\Esf}$, which we further identify with
  \[
   \Vals{\Esf} \coloneqq \{\sum n_x\cdot x \in \Div(\Esf) \mid \sum n_x = 0\}
  \]
 via $y \mapsto \restr{y}{E}$. We denote by $[\rm div]\colon \breve\Esf \to \Div^0(\breve\Esf)_{\sigma_\Esf} = \Vals{\Esf}$ the composition of ${\rm div}$ with the canonical projection $\Vals{\breve\Esf}\epi (\Vals{\breve\Esf})_{\sigma_\Esf}$. 
 
 \subsection{}For any finite extension $\Esf/\Fsf$, we denote by $f_{\Esf/\Fsf}\colon C_{\breve\Esf} \to C_{\breve\Fsf}$ the morphism of curves corresponding to $\breve\Esf/\breve\Fsf$. Using the identification above, the push-forward and pull-back of divisors on these curve induces morphisms 
 $f_{\Esf/\Fsf}^\ast\colon \Vals{\Fsf} \to \Vals{\Esf}$ and $f_{\Esf/\Fsf,\ast}\colon \Vals{\Fsf} \to \Vals{\Esf}$, More explicitly, these morphisms are given by
 \begin{align*} 
 f_{\Esf/\Fsf}^\ast(x) &= \sum_{y|x} [E_y:F_x] \cdot  y \\    
 f_{\Esf/\Fsf,\ast}(y) &= \restr{y}{\Fsf}.
\end{align*}
  Note that when $\Esf/\Fsf$ is Galois, $f_{\Esf/\Fsf,\ast}$ induces an isomorphism ${\Vals{\Esf}}_{\Gal(\Esf/\Fsf)} \cong \Vals{\Fsf}$.

\section{Tori} \label{sect-tori}
 
 Let $\Esf/\Fsf$ be a finite Galois extension. In order to generalise the isomorphism in Example~\ref{exa-Gm} to arbitrary $\Fsf$-tori, we rewrite it as 
 \begin{equation} \label{eq-B-for-restriction-of-scalars}
   (X_\ast(\Res_{\Esf/\Fsf}\GG_m) \otimes \Vals{\Esf})_{\Gal(\Esf/\Fsf)} \cong \B(\Fsf,\Res_{\Esf/\Fsf}\GG_m),
 \end{equation}
 via the canonical isomorphism of functors
  \[(X_\ast(\Res_{\Esf/\Fsf}\GG_m) \otimes (\cdot))_{\Gal(\Esf/\Fsf)} \cong (\ZZ[\Gal{(\Esf/\Fsf)}] \otimes (\cdot))_{\Gal(\Esf/\Fsf)} \cong (\cdot)\]
 
 \begin{prop} \label{prop-kottwitz-pt-for-tori}
  Let $\Esf/\Fsf$ be a finite Galois extension. The isomorphism (\ref{eq-B-for-restriction-of-scalars}) can be extended uniquely to an isomorphism of functors
  \[
   (\Vals{\Esf}\otimes X_\ast(\cdot))_{\Gal(\Esf/\Fsf)} \isom \B(\Fsf,\cdot)
  \]
  of $\Fsf$-tori which split over $\Esf$. Moreover, for any Galois extension $\Esf \subset \Esf'$ the diagram of functors on $\Esf$-split $\Fsf$-tori
  \begin{center}
  \begin{tikzcd}
   (\Vals{\Esf'}\otimes X_\ast(\cdot))_{\Gal(\Esf'/\Fsf)} 
   \arrow{d}{\sim}[swap]{f_{\Esf'/\Esf,\ast}} 
   \ar{r}{\sim} 
   & \B(\Fsf,\cdot) 
   \ar[equal]{d}\\
   (\Vals{\Esf}\otimes X_\ast(\cdot))_{\Gal(\Esf/\Fsf)}   
   \ar{r}{\sim} 
   & \B(\Fsf,\cdot)
  \end{tikzcd}  
  \end{center}
  commutes.
 \end{prop}
 \begin{proof}
 We first show the second part of the proposition, assuming that the first part holds. As the diagram
  \begin{center}
   \begin{tikzcd}
    \Vals{\Esf'} \ar{r}{\sim} \ar[two heads]{d}{f_{\Esf'/\Esf,\ast}} & \B(\Fsf,\Res_{\Esf'/\Fsf}\GG_m) \ar{d}{\Nm_{\Esf'/\Esf}} \\
    \Vals{\Esf} \ar{r}{\sim} & \B(\Fsf,\Res_{\Esf/\Fsf}\GG_m)
   \end{tikzcd}
  \end{center}
  commutes, the functor homomorphism extending (\ref{eq-B-for-restriction-of-scalars}) for $\Esf'$ also yields the isomorphism for $\Esf$. Hence the second part of the proposition holds by uniqueness.
 
  The functor $X_\ast$ is represented by $\Res_{\Esf/\Fsf} \GG_m$. Thus
  \begin{align*}
   \Hom(\Vals{\Esf} \otimes X_\ast(\cdot),\B(\Fsf,\cdot)) &\cong \Hom(X_\ast(\cdot),\Hom(\Vals{\Esf},\B(\Fsf,\cdot))) \\
   &\cong \Hom(\Vals{\Esf},\B(\Fsf,\Res_{\Esf/\Fsf} \GG_m)) \\
   &\cong \End(\Vals{\Esf})
  \end{align*}
  Because the second isomorphism is induced by evaluating at $\Res_{\Esf/\Fsf}\GG_m$, any such functor homomorphism is uniquely determined by its values on the $\Res_{\Esf/\Fsf} \GG_m$-valued points, proving the uniqueness part of the proposition. One checks easily that the functor homomorphism corresponding to the identity extends the map $X_\ast(\Res_{\Esf/\Fsf}\GG_m) \otimes \Vals{\Esf} \to \B(\Res_{\Esf/\Fsf}\GG_m,\Fsf)$ induced by (\ref{eq-B-for-restriction-of-scalars}). Note that every element $\gamma \in \Gal(\Esf/\Fsf)$ induces an endomorphism $X_\ast(\cdot) \isom X_\ast(\cdot)$. Hence $\Gal(\Esf/\Fsf)$ acts on $\Hom(X_\ast(\cdot) \otimes \Vals{\Esf},\B(\Fsf,\cdot))$ via precomposing. By tracing through the definitions, one checks that the action corresponds to the standard $\Gal(\Fsf/\Esf)$-action on $\End(\Vals{\Esf})$. In particular, the functor morphism corresponding to the identity is $\Gal(\Fsf/\Esf)$-invariant and thus induces a functor
  \[
   (\Vals{\Esf}\otimes X_\ast(\cdot))_{\Gal(\Esf/\Fsf)} \to \B(\Fsf,\cdot)
  \]    
  It remains to prove that this is an isomorphism. Fix an $\Esf$-split torus $\Tsf$ and an exact sequence
  \begin{center}
   \begin{tikzcd}
    1 \ar{r} & \Usf \ar{r} & \Ssf \ar{r} & \Tsf \ar{r} & 1,
   \end{tikzcd}
  \end{center}
  where $\Ssf$ is a product of copies of $\Res_{\Esf/\Fsf} \GG_m$. We obtain a commutative diagram with exact rows
  \begin{center}
   \begin{tikzpicture}[baseline= (a).base]
    \node[scale=0.8] (a) at (0,0){
    \begin{tikzcd}[column sep = small]
     & (X_\ast(\Usf) \otimes \Vals{\Esf})_{\Gal(\Fsf/\Esf)} \ar{r} \ar{d} & (X_\ast(\Ssf) \otimes \Vals{\Esf})_{\Gal(\Fsf/\Esf)} \ar{r} \ar{d} & (X_\ast(\Tsf) \otimes \Vals{\Esf})_{\Gal(\Fsf/\Esf)} \ar{r} \ar{d} & 1 \\
      & \B(\Fsf,\Usf) \ar{r} & \B(\Fsf,\Ssf) \ar{r} & \B(\Fsf,\Tsf) \ar{r} & 1.
    \end{tikzcd}} ;
   \end{tikzpicture}
  \end{center}
  Since the middle vertical morphism is an isomorphism the right map is surjective. Since this must be true for any $\Esf$-torus, it also holds for $\Usf$. So the left vertical map is surjective, proving that the right map is also injective by the $4$-lemma.
 \end{proof}
 
  To simplify notation, we define the functor from $\Fsf$-tori to abelian groups
  \[
   \A(\Fsf,\cdot) \coloneqq \varprojlim_{f_{\Esf'/\Esf,\ast}} (\Vals{\Esf}\otimes X_\ast(\cdot))_{\Gal(\Esf/\Fsf)}.
  \]
  Thus the main result of the Proposition~\ref{prop-kottwitz-pt-for-tori} above is that $\A(\Fsf,\cdot)$ and $\B(\Fsf,\cdot)$ are canonically isomorphic to each other.
  
 \begin{defn} \label{defn-Kottwitz-pt-tori}
  For any $\Fsf$-torus $\Tsf$ we define  $\bar\kappa_\Tsf\colon \B(\Fsf,\Tsf) \to \A(\Fsf,\Tsf)$ as the inverse of above isomorphism $\A(\Fsf,\Tsf) \cong \B(\Fsf,\Tsf)$.
 \end{defn}

 \subsection{} \label{ss-maps} Let $\Fsf'/\Fsf$ be a finite Galois extension. We define the morphism $\Nm_{\Fsf'/\Fsf}\colon \A(\Fsf,\cdot) \to \A(\Fsf',\cdot)$ as the morphism induced by restriction morphisms $(\Vals{\Esf}\otimes X_\ast(\cdot))_{\Gal(\Esf/\Fsf)} \to (\Vals{\Esf}\otimes X_\ast(\cdot))_{\Gal(\Esf/\Fsf')}, a \mapsto \sum_{\gamma } \gamma \cdot a, $
  where $\gamma$ runs through a set of representatives of $\Gal(\Esf/\Fsf)/\Gal(\Esf/\Fsf')$.  Similarly, we let $\cores_{\Fsf'/\Fsf}\colon \A(\Fsf',\cdot) \to \A(\Fsf,\cdot)$  to be induced by the canonical projection $ (\Vals{\Esf}\otimes X_\ast(\cdot))_{\Gal(\Esf/\Fsf')} \epi (\Vals{\Esf}\otimes X_\ast(\cdot))_{\Gal(\Esf/\Fsf)}$.
  Moreover, for any place $x \in |C|$, we define $\loc_x\colon \A(\Fsf,\cdot) \to X_\ast(\cdot)_{\Gal(\scl{F_x},F_x)}$ as follows. Let $\Tsf$ be an $\Fsf$-torus with splitting field $\Esf$. Denoting by $y_x|x$ the place defined by our chosen embedding $\scl{\Fsf} \mono \scl{F_x}$, the field $E_{y_x}$ is a splitting field of $\Tsf_{F_x}$. We define $\loc_x$ as the composition of
  \[
   (\Vals{\Esf} \otimes X_\ast(\Tsf))_{\Gal(\Esf/\Fsf)} \xrightarrow{\loc_{\Esf,x} \otimes \id} (\bigoplus_{y|x} \ZZ\cdot y \otimes X_\ast(\Tsf))_{\Gal(\Esf/\Fsf)} \cong X_\ast(\Tsf_{F_x})_{\Gal(E_{y_x}/F_x)} = X_\ast(\Tsf_{F_x})_{\Gal(\scl{F}_x/F_x)}.
  \]
 
 \begin{lem} \label{lem-Kottwitz-pt-field-ext}  
  The isomorphism $\A(\Fsf,\cdot) \cong \B(\Fsf,\cdot)$ is compatible with the morphism defined above. More precisely, the following holds.
  \begin{subenv}
   \item Let $\Fsf'/\Fsf$ be a finite field extension. Then the diagrams
   \begin{center}
    \begin{tikzcd}
     \B(\Fsf,\cdot) \ar{d}{(\ref{B-field-extn})} \ar{r}{\sim} & \A(\Fsf,\cdot)  \ar{d}{\Nm_{\Fsf'/\Fsf}}  &
     \B(\Fsf',\cdot) \ar{d}{(\ref{B-cores})} \ar{r}{\sim} & \A(\Fsf',\cdot)  \ar{d}{\cores_{\Fsf'/\Fsf}}   \\
     \B(\Fsf',\cdot) \ar{r}{\sim}  & \A(\Fsf',\cdot)  & 
     \B(\Fsf,\cdot) \ar{r}{\sim} & \A(\Fsf,\cdot)  
    \end{tikzcd}
   \end{center}
   commute.   
   \item Let $x \in |C|$. Then the diagram
   \begin{center}
    \begin{tikzcd}
     \B(\Fsf,\cdot) \ar{d}{(\ref{B-localisation})}  \ar{r}{\sim}  & \A(\Fsf,\cdot) \ar{d}{\loc_{x}} \\
     \B(F_x,\cdot) \ar{r}{\sim}  & X_\ast(\cdot)_{\Gal(\Esf/\Fsf)}
    \end{tikzcd}
   \end{center}
   commutes, where the isomorphism in the bottom row is given by \cite[\S~2.4]{Kottwitz:Gisoc1}.
  \end{subenv}
 \end{lem} 
 \begin{proof}
  It suffices to prove the claim for $\Tsf = \Res_{\Esf/\Fsf}(\GG_m)$, which we show by direct calculation. Using the above identification $A(\Fsf,\Res_{\Esf/\Fsf} \GG_m) = \Vals{\Esf}$, the isomorphism $\bar\kappa_\Tsf$ is induced by the composition
  \[
   \Tsf(\breve\Fsf) = (\Esf \otimes \breve\Fsf)^\times \xrightarrow{\Nrm^{(d)} \circ m} \breve\Esf^\times \xrightarrow{[\rm div]} \Vals{\Esf},
  \]
  where $m$ denotes the multiplication map and $d = [k_\Esf:k_\Fsf]$.
  
  To prove (1), we can assume that $\Esf$ is big enough such that $\Fsf' \subset \Esf$. In particular, we can identify $\Esf \otimes_\Fsf \Fsf' \cong \Esf^I$ here $I$ is the set of $\Fsf$-linear embeddings $\Fsf' \mono \Esf$. This yields an isomorphism  $\Tsf_{\Fsf'} \cong (\Res_{\Esf/\Fsf'} \GG_m)^I$ and thus $\A(\Fsf',\Tsf) = \Vals{\Esf}^I$. Using these identifications, $\Nm_{\Fsf'/\Fsf}\colon \Vals{\Esf} \to \Vals{\Esf}^I$ is the diagonal embedding and $\cores_{\Fsf'/\Fsf}\colon \Vals{\Esf}^I \to \Vals{\Esf}$ maps each $I$-tuple to the sum of its components. We denote $d' \coloneqq [k_{\Fsf'}:k_\Fsf]$ and $d'' \coloneqq [k_\Esf : k_{\Fsf'}]$. Then (1) follows from the commutativity of
  \begin{center}
   \begin{tikzcd}
    \Tsf(\breve\Fsf) \ar[equal]{r} \ar{d}{\Nrm^{(d')}} & (\Esf \otimes_\Fsf \breve\Fsf)^\times \ar{rr}{\Nrm^{(d)} \circ m} \ar{d}{\id \otimes \Nrm^{(d')}} & & \breve\Esf^\times \ar{r}{[{\rm div}]} \ar[hook]{d}{diag.}& \Vals{\Esf} \ar{d}{\Nrm_{\Fsf'/\Fsf}} \\
    \Tsf(\breve\Fsf') \ar[equal]{r} & (\Esf \otimes_\Fsf \breve\Fsf')^\times \ar{r}{\sim} & (\Esf^I \otimes_{\Fsf'} \breve\Fsf')^\times \ar{r}{\Nrm^{(d'')} \circ m} & (\breve\Esf^\times)^I \ar{r}{[{\rm div}]} & \Vals{\Esf}^I
   \end{tikzcd}
  \end{center}  
  and 
  \begin{center}
   \begin{tikzcd}
     \Tsf(\breve\Fsf') \ar[equal]{r} \ar{d}{g \mapsto \prod \tau(g)} & (\Esf \otimes_\Fsf \breve\Fsf')^\times \ar{r}{\sim} \ar{d}{\id \otimes \Nm_{\breve\Fsf'/\breve\Fsf}} & (\Esf^I \otimes_{\Fsf'} \breve\Fsf')^\times \ar{r}{\Nrm^{(d'')} \circ m} & (\breve\Esf^\times)^I \ar{r}{[{\rm div}]} \ar{d}{mult.} & \Vals{\Esf}^I \ar{d}{\cores_{\Fsf'/\Fsf}} \\
     \Tsf(\breve\Fsf) \ar[equal]{r} & (\Esf \otimes_\Fsf \breve\Fsf)^\times \ar{rr}{\Nrm^{(d)} \circ m}  & & \breve\Esf^\times \ar{r}{[{\rm div}]} & \Vals{\Esf},
   \end{tikzcd}
  \end{center}
  where in the leftmost vertical morphism the product runs over all $\breve\Fsf$-linear embeddings $\tau\colon\breve\Fsf' \mono \scl{\Fsf}$.
  
  For (2) we identify $\Esf \otimes_{\Fsf} F_x \cong \prod_{y|x} E_y$, inducing $\Tsf_{F_x} \cong \prod_{y|x} \Res_{E_y/F_x} \GG_m$ and thus $X_\ast(\Tsf_{F_x})_{\Gal(\scl{F_x}/F_x)} = \prod_{y|x} \ZZ \cdot y$. Under these isomorphisms $\loc_x$ is identified with $\loc_{\Esf,x}$. We denote $d' = \deg(x)$ and $d'' = [k(y):k(x)]$. Then (2) follows from the commutativity of

 \begin{center}
   \begin{tikzcd}
    \Tsf(\breve\Fsf) \ar[equal]{r} \ar{d}{\Nrm^{(d')}} & (\Esf \otimes_\Fsf \breve\Fsf)^\times \ar{rr}{\Nrm^{(d)} \circ m} \ar{d}{\id \otimes \Nrm^{(d')}} & & \breve\Esf^\times \ar{r}{[{\rm div}]} \ar[hook]{d}{diag.}& \Vals{\Esf} \ar{d}{\loc_x} \\
    \Tsf(\breve{F_x}) \ar[equal]{r} & (\Esf \otimes_{\Fsf} \breve{F_x})^\times \ar{r}{\sim} & \prod_{y|x} (E_y \otimes_{F_x} \breve{F_x})^\times \ar{r}{N^{(d'')}\circ m} & \prod_{y|x} \breve{E_y}^\times \ar{r}{\oplus \val_y} & \bigoplus_{y|x} \ZZ.
   \end{tikzcd}
  \end{center}

 \end{proof} 
 \subsection{} \label{sect-Kottwitz-pt}
  As a consequence of previous lemma, $\bar\kappa_{\Tsf}([\bsf])$ depends only on $\bsf \in \Tsf(\breve\Fsf)$ and not on the choice of the base field $\Fsf$. More precisely, by Lemma~\ref{lem-Kottwitz-pt-field-ext}~(1), we see that for any intermediate extension $\Fsf' \subset \breve\Fsf$ the diagram
  \begin{center}
   \begin{tikzcd}
    \Tsf(\breve\Fsf) \ar[equal]{d} \ar{r} & \B(\Fsf',\Tsf) \ar{r}{\bar\kappa_{\Tsf_{\Fsf'}}} & \A(\Fsf',\Tsf) \ar[two heads]{d}{\cores_{\Fsf'\Fsf}} \\
    \Tsf(\breve\Fsf) \ar{r} & \B(\Fsf,\Tsf) \ar{r}{\bar\kappa_\Tsf} & \A(\Fsf,\Tsf)
   \end{tikzcd}
  \end{center}
  commutes.  By taking the limit over all such $\Fsf'$ we obtain a $\sigma$-equivariant morphism
  \[
   \kappa_\Tsf\colon \Tsf(\breve\Fsf) \to \A(\breve\Fsf,\Tsf) \coloneqq \varprojlim_{\Fsf\subset \tilde\Fsf \subset \breve\Fsf} \A(\tilde\Fsf,\Tsf) 
  \]
  such that for every $\bsf \in \Tsf(\breve\Fsf)$ we obtain $\bar\kappa_\Tsf([\bsf])$ by taking the image of $\kappa_\Tsf(\bsf)$ under the canonical projection $\A(\breve\Fsf,\Tsf) \epi \A(\Fsf,\Tsf)$. We call $\kappa_\Tsf(\bsf)$ the Kottwitz point of $\bsf$.
  Note that the definition of $A(\Fsf,\Tsf)$ does not change when we extend the projective system by allowing $\Esf$ to be any function field of a subfield of $k$ over $\Fsf$. Since projective limits commute with each other as well as with taking coinvariants commute we may write
  \[
    A(\breve\Fsf,\Tsf) = \varprojlim_{\tilde\Fsf/\Fsf \textnormal{ finite}} \varprojlim_{f_{\Esf'/\Esf,\ast}} (\Vals{\Esf}\otimes X_\ast(\Tsf))_{\Gal(\Esf/\tilde\Fsf)} = \varprojlim_{f_{\Esf'/\Esf,\ast}} (\Vals{\Esf}\otimes X_\ast(\Tsf))_{\Gal(\Esf/\breve\Fsf)},
  \]
  where in the rightmost limit runs $\Esf$ and $\Esf'$ run over finite extensions of $\breve\Fsf$.
  
 We obtain the following analogue of Lemma~\ref{lem-Kottwitz-pt-field-ext} for $A(\breve\Fsf,\Tsf)$.
 
  \begin{lem} \label{lem-Kottwitz-pt-field-ext-2}
  Let $\Tsf$ be an $\Fsf$-torus.
  \begin{subenv}
   \item Let $\Fsf'/\Fsf$ be a finite field extension and let $\Tsf' = \Tsf_\Fsf$. Then the diagrams
   \begin{center}
    \begin{tikzcd}
     \Tsf(\breve\Fsf) \ar[hook]{d} \ar{r}{\kappa_\Tsf} & \A(\breve\Fsf,\Tsf)  \ar{d}{\Nm_{\breve\Fsf'/\breve\Fsf}}  &
     \Tsf(\breve\Fsf') \ar{d}{(\ref{B-cores})} \ar{r}{\kappa_{\Tsf'}} & \A(\breve\Fsf',\Tsf)  \ar{d}{\cores_{\breve\Fsf'/\breve\Fsf}}   \\
     \Tsf(\breve\Fsf') \ar{r}{\kappa_{\Tsf'}}  & \A(\breve\Fsf',\Tsf)  & 
     \Tsf(\breve\Fsf) \ar{r}{\kappa_{\Tsf}} & \A(\breve\Fsf,\Tsf)  
    \end{tikzcd}
   \end{center}
   commute.   
   \item Let $x \in |C|$ and denote $T_x = \Tsf_{F_x}$. Then the diagram
   \begin{center}
    \begin{tikzcd}
     \Tsf(\breve\Fsf) \ar[hook]{d}  \ar{r}{\kappa_\Tsf}  & \A(\breve\Fsf,\Tsf) \ar{d}{\loc_{x}} \\
     T_x(\Fbreve_x) \ar{r}{\kappa_{T_x}}  & X_\ast(\cdot)_{\Gal(\scl{F_x}/\Fbreve_x)}(T_x)
    \end{tikzcd}
   \end{center}
   commutes, where the lower morphism is constructed in \cite[\S~7]{Kottwitz:Gisoc2}.
   \end{subenv}
  \end{lem}
  \begin{proof}
   The follows by taking the limit over $\Fsf \subset \tilde\Fsf \subset \breve\Fsf$ for the diagrams in Lemma~\ref{lem-Kottwitz-pt-field-ext}, setting $\tilde\Fsf' = \tilde\Fsf \cdot \Fsf'$ in the first part of the lemma. Note that for $\tilde\Fsf$ big enough the fields $\tilde\Fsf$ and $\tilde\Fsf'$ have the same fields of constants, thus the $\sigma$-twisted powers occurring in Lemma~\ref{lem-Kottwitz-pt-field-ext} are trivial.
  \end{proof}
  
 \subsection{} 
 In order to construct the Newton point, we consider the pro-tori 
 \begin{align*}
  \DD_{\Esf/\Fsf} &= \Spec \Fsf[\Vals{\Esf}] \\
  \DD_{\Fsf} &= \varprojlim \DD_{\Esf/\Fsf},
 \end{align*}
 that is $\DD_\Fsf$ is the protorus with character group $\Vals{\scl{\Fsf}} \coloneqq \varinjlim_{f_{\Esf/\Fsf}^\ast} \Vals{\Esf}$. For any place $y$ of $\scl{\Fsf}$ over a place $x$ of $\Fsf$ the morphisms
 \[
 {\rm ev}_{\Esf,y}\colon \Vals{\Esf} \to \QQ, \sum n_{y'} \cdot y' \mapsto [E_{y}:F_x]\iv \cdot n_{y_x}.
 \]
 are compatible with $f^\ast_{\Esf'/\Esf}$ and hence induce a morphism ${\rm ev}_x\colon \Vals{\scl{\Fsf}} \to \QQ$. We denote by $\iota_y\colon \DD \to \DD_\Fsf$ the corresponding morphism of protori, where $\DD$ denotes the protorus with character group $\QQ$. Note that the morphism
 \[
  {\rm ev}\colon \Vals{\scl\Fsf} \to \prod_y \QQ
 \]
 is injective, hence any morphism $\nu\colon \DD_\Fsf \to \Tsf$ is uniquely determined by the family $(\nu \circ \iota_y)_y$. Recall that we fixed a place $y_x|x$ of $\scl{\Fsf}$ in the introduction. We denote ${\rm ev}_x \coloneqq {\rm ev}_{y_x}$ and $\iota_x \coloneqq \iota_{y_x}$.
  
 We claim that for any $\Fsf$-torus $\Tsf$ and any finite Galois extension of splitting fields $\Esf'/\Esf$ the diagram
 \begin{center}
  \begin{tikzcd}
   (\Vals{\Esf'} \otimes X_\ast(\Tsf))_{\Gal(\Esf'/\Fsf)} \ar{d}{f_{\Esf'/\Esf,\ast} \otimes \id}[swap]{\sim} \ar{r}{\Nm_{\Esf'/\Fsf}} & (\Vals{\Esf'} \otimes X_\ast(\Tsf))^{\Gal(\Esf'/\Fsf)} \\
   (\Vals{\Esf} \otimes X_\ast(\Tsf))_{\Gal(\Esf/\Fsf)} \ar{r}{\Nm_{\Esf/\Fsf}} & (\Vals{\Esf} \otimes X_\ast(\Tsf))^{\Gal(\Esf/\Fsf)} \ar{u}[swap]{f_{\Esf'/\Esf}^\ast \otimes \id}
  \end{tikzcd}
 \end{center} 
 commutes, where the vertical arrows are mapping a coset to the sum of its elements. Indeed, as the norm map for $\Esf'/\Fsf$ is the composition of the Norm maps for $\Esf'/\Esf$ and $\Esf/\Fsf$, one can reduce to the case $\Esf = \Fsf$. In particular $\Tsf$ is split, so we may assume $\Tsf = \GG_m$, i.e.\ $X_\ast(\Tsf) = \ZZ$. The claim now follows by construction. 
 Passing to the limit, we obtain a morphism of functors
 \begin{equation} \label{Kottwitz-to-Newton}
  \A(\Fsf,\cdot) \to (\Vals{\scl{\Fsf}} \otimes X_\ast(\cdot))^{\Gal(\scl\Fsf/\Fsf)} = \Hom_\Fsf(\DD_\Fsf,\cdot).
 \end{equation}
 Note that the Norm map induces an isomorphism $\A(\Fsf,\cdot)_\QQ \isom \Hom_\Fsf(\DD_\Fsf,\cdot)_\QQ$, in particular the kernel of (\ref{Kottwitz-to-Newton}) equals $\A(\Fsf,\cdot)_{\rm tors}$.
 
 \begin{defn}
  For any $\Fsf$-torus $\Tsf$ and $\bsf \in \Tsf(\breve\Fsf)$, we define its Newton point $\nu_\Tsf(\bsf)$ as the image of $\bar\kappa_\Tsf([\bsf])$ under (\ref{Kottwitz-to-Newton}) above.
 \end{defn}
 
 \begin{lem} \label{lem-Newton-pt-field-ext}
  Let $\Tsf$ be an $\Fsf$-torus and $\bsf \in \Tsf(\Fsf)$. We fix a finite extension field $\Fsf'/\Fsf$ and a point $x \in |C|$ and denote by $\Tsf'$ and $T_x$ the respective base change of $\Tsf$ to $\Fsf'$ and $F_x$. Then the following holds for all $\bsf \in \Tsf(\breve\Fsf)$.
  \begin{subenv}
   \item $\nu_{\Tsf'}({\rm N}^{([k_{\Fsf'}:k_\Fsf])}(\bsf))=  \nu_{\Tsf}(\bsf)$
   \item $\nu_{\Tsf}(\Nm_{\breve\Fsf'/\breve\Fsf}\bsf) = \Nm_{\breve\Fsf'/\breve\Fsf}\nu_{\Tsf'}(\bsf)$ and
   \item $\nu_{T_x}(N^{(\deg x)}\bsf) = \nu_{\Tsf}(\bsf) \circ \iota_x$.
  \end{subenv}
 \end{lem}
 \begin{proof}
  We fix $\Esf/\Fsf$ splitting $\Tsf$ and containing $\Fsf'$. The above assertions follow by Lemma~\ref{lem-Kottwitz-pt-field-ext} and commutativity of the diagrams
  \begin{center}
   \begin{tikzcd}
    (\Vals{\Esf} \otimes X_\ast(\Tsf))_{\Gal(\Esf/\Fsf)} \ar{r}{\Nm_{\Esf/\Fsf}} \ar{d}{\Nm_{\Fsf'/\Fsf}} & (\Vals{\Esf} \otimes X_\ast(\Tsf))^{\Gal(\Esf/\Fsf)} \ar[equal]{d}  \\
    (\Vals{\Esf} \otimes X_\ast(\Tsf))_{\Gal(\Esf/\Fsf')} \ar{r}{\Nm_{\Esf/\Fsf'}} & (\Vals{\Esf} \otimes X_\ast(\Tsf))^{\Gal(\Esf/\Fsf')} \\
   
    (\Vals{\Esf} \otimes X_\ast(\Tsf))_{\Gal(\Esf/\Fsf')} \ar{r}{\Nm_{\Esf/\Fsf}} \ar{d}{\cores_{\Fsf'/\Fsf}} & (\Vals{\Esf} \otimes X_\ast(\Tsf))^{\Gal(\Esf/\Fsf')} \ar{d}{\Nm_{\Fsf'/\Fsf}}  \\
    (\Vals{\Esf} \otimes X_\ast(\Tsf))_{\Gal(\Esf/\Fsf)} \ar{r}{\Nm_{\Esf/\Fsf}} & (\Vals{\Esf} \otimes X_\ast(\Tsf))^{\Gal(\Esf/\Fsf)} \\
    
   (\Vals{\Esf} \otimes X_\ast(\Tsf))_{\Gal(\Esf/\Fsf)} \ar{r}{\Nm_{\Esf/\Fsf}} \ar{d}{\loc_x} & (\Vals{\Esf} \otimes X_\ast(\Tsf))^{\Gal(\Esf/\Fsf)} \ar{d}{{\rm ev}_{E,x}}  \\
    ( X_\ast(\Tsf))_{\Gal(E_{y_x}/F_x)} \ar{r}{[E_{y_x}:F_x]\iv \Nm_{E_{y_x}/F_x}} & X_\ast(\Tsf)_\QQ^{\Gal(E_{y_x}/F_x)}. 
   \end{tikzcd}
  \end{center}
  The commutativity of the first two diagrams follows directly from the definition of the norm map. Finally we evaluate the third diagram at an element $ [ \sum_{y' \in |C_E|} y \otimes \lambda_{y}] \in  (\Vals{\Esf} \otimes X_\ast(\Tsf))_{\Gal(\Esf/\Fsf)}$:
  \begin{center}
   \begin{tikzcd}
    \left[ \sum_{y' \in |C_E|} y \otimes \lambda_{y}\right]  \arrow[mapsto]{r} \arrow[mapsto]{d} & \sum_{y \in |C_E|} y \otimes (\sum_{\tau \in \Gal(\Esf/\Fsf)} \tau\iv(\lambda_{\tau(y')}) \arrow[mapsto]{d} 
    \\
    \left[ \sum_{\tau \in \Gal(\Esf/\Fsf)/\Gal(E_{y_x}/F_x)} \tau\iv(\lambda_{\tau(y)})\right] \ar[mapsto]{r} & \left[E_{y_x}:F_x\right]\iv \cdot \sum_{\tau \in \Gal(\Esf/\Fsf)}\tau\iv(\lambda_{\tau(y)}),
   \end{tikzcd}
  \end{center}
  proving the commutativity of the third diagram.
 \end{proof}
 
 \begin{rmk} \label{rmk-choice}
  Since $\nu_{T_x}(\bsf)$ is independent of the choice of $y_x|x$, one would expect the same to be true for $\nu_{\Tsf}(\bsf) \circ \iota_x$. However, one has to be careful with this notion. Consider two different choices of embeddings $i,i'\colon \scl\Fsf \mono \scl{F_x}$ corresponding to places $y_x$ and $y_x'$ of $\scl\Fsf$, respectively. Let $\tau \in \Gal(\scl\Fsf/\Fsf)$ such that $i' = i \circ \tau$ and denote by $\iota_x,\iota_{x'}\colon \DD \to \DD_F$ the respective embeddings corresponding to $y_x$ and $y_x'$. Then $\iota_x' =  \tau(\iota_x)$, thus  
  \[
    \nu_{\Tsf}(\bsf) \circ \iota_x' = \tau(\nu_{\Tsf}(\bsf) \circ \iota_x).
  \]
  In particular, $\nu_{\Tsf}(\bsf) \circ \iota_x$, as a morphism over $\scl{\Fsf}$ does depend on the choice of $y_x$. However, applying $i$ to this equation shows that
  \[
   i'(\nu_{\Tsf}(\bsf) \circ \iota_x') = i(\nu_{\Tsf}(\bsf) \circ \iota_x),
  \]
  i.e. $\nu_{\Tsf}(\bsf) \circ \iota_x$ is independent of choice when considered over $\scl{F_x}$.
 \end{rmk} 
 
\section{The Kottwitz and Newton point for reductive groups} \label{sect-redgps}

 For any (connected) reductive group $\Gsf$ over $\Fsf$, we denote by $\pi_1(\Gsf)$ Borovoi's fundamental group. That is if $\Tsf$ is a maximal torus of $\Gsf$ and $Q^\vee \subset X_\ast(\Tsf)$ denotes the absolute coroot lattice then $\pi_1(\Gsf)$ is defined as the Galois module $X_\ast(\Tsf)/Q^\vee$. This construction is independent of the choice of $\Tsf$ up to canonical isomorphism and $\pi_1(\cdot)$ is an exact functor (see \cite[\S~1]{Borovoi:AbCohRedGp})

  Motivated by the previous chapter we define
 \begin{align*}
   \A(\Fsf,\Gsf) &\coloneqq \varprojlim_{f_{\Esf'/\Esf,\ast}} (\Vals{\Esf}\otimes \pi_1(\Gsf))_{\Gal(\Esf/\Fsf)}, \\
   \A(\breve\Fsf,\Gsf) &\coloneqq \varprojlim_{f_{\Esf'/\Esf,\ast}} (\Vals{\Esf}\otimes \pi_1(\Gsf))_{\Gal(\Esf/\breve\Fsf)},
 \end{align*}
 where the limit ranges over all finite $\Esf/\Fsf$ (and finite $\Esf/\breve\Fsf$ resp.) such that $\Gsf_{\Esf}$ is split. The construction of the localisation map in (\ref{ss-maps}) generalises to
  \begin{align*}
   \loc_x\colon \A(\breve\Fsf,\Gsf) &\to \pi_1(\Gsf)_{\Gal(\scl{\breve F_x}/\breve F_x)}.
 \end{align*}
 
  In the following we define the Kottwitz point $\kappa_\Gsf\colon G(\breve\Fsf) \to A(\breve\Fsf,\Gsf)$ by canonically extending the functor on tori defined in the previous chapter.
 
 \begin{prop}
  There exists a unique family of homomorphisms $\kappa_\Gsf\colon G(\breve\Fsf) \to A(\breve\Fsf,\Gsf)$ which is functorial in $\Gsf$ and coincides with the definition in section~\ref{sect-Kottwitz-pt} when restricted to the case that $\Gsf$ is a torus.
 \end{prop}
 \begin{proof}
  We follow the proof of \cite[Prop.~9.1]{Kottwitz:Gisoc3}. We first show that $\kappa_\Gsf$ extends uniquely to reductive groups with simply connected derived group $\Gsf^{\der}$. Then the canonical projection $\Gsf \epi \Gsf/\Gsf^{\der} \eqqcolon \Dsf$ induces an isomorphism $\pi_1(\Gsf) \isom \pi_1(\Dsf) = X_\ast(\Dsf)$.  By functoriality, $\kappa_\Gsf$ has to be the unique homomorphism making the diagram
  \begin{center}
   \begin{tikzcd}
    \Gsf(\breve\Fsf) \arrow{r} \arrow[dashed]{d}{\kappa_\Gsf} & \Dsf(\breve\Fsf) \arrow{d}{\kappa_\Dsf} \\
    \A(\breve\Fsf,\Gsf) \arrow{r}{\sim}& \A(\breve\Fsf,\Dsf)
   \end{tikzcd}
  \end{center}
  commute. As $\Gsf \to \Gsf/\Gsf^\der$ is functorial in $\Gsf$, the functoriality of $\kappa_{\Gsf}$ follows from its functoriality on tori.
  
  For general reductive $\Gsf$ consider a $z$-extension
  \begin{center}
   \begin{tikzcd}
    1 \ar{r} & \Zsf \ar{r} & \Gsf_1 \ar{r} & \Gsf \ar{r} & 1
   \end{tikzcd}
  \end{center}
  with $\Gsf_1^{\der}$ simply connected. Thus $\kappa_\Gsf$ must be the unique morphism making the diagram
  \begin{center}
   \begin{tikzcd}
   1 \ar{r} & \Zsf(\breve\Fsf) \ar{r} \ar{d}{\kappa_{\Zsf}} &\Gsf_1(\breve\Fsf) \ar{r} \ar{d}{\kappa_{\Gsf_1}}& \Gsf(\breve\Fsf) \ar{r} \ar[dashed]{d}{\kappa_\Gsf} & 1 \\
   & \A(\breve\Fsf,\Zsf) \ar{r} & \A(\breve\Fsf,\Gsf_1) \ar{r} & \A(\breve\Fsf,\Gsf) \ar{r} & 1
   \end{tikzcd}
  \end{center}
  commute. As any $z$-extension yields a short exact sequence of fundamental groups, the rows of above diagram are exact and thus $\kappa_\Gsf$ exists. As a consequence of \cite[Lemma~2.4.4]{Kottwitz:StTrFormCuspTemp} the morphism $\kappa_\Gsf$ does not depend on the choice of a $z$-extension as above and its functoriality follows from the functoriality of $\kappa_{\Gsf_1}$.
 \end{proof}
 
 \begin{cor} \label{cor-Kottwitz-pt-functoriality}
  Let $\bsf \in \Gsf(\Fsf)$.
  \begin{subenv}
   \item For any finite field extension $\Fsf'/\Fsf$ we have
   \[
    \kappa_{\Gsf_{\Fsf'}}(\bsf) = \Nm_{\breve\Fsf'/\breve\Fsf} (\kappa_\Gsf(\bsf))
   \]
   \item For any $x \in |C|$ we have
   \[
    \kappa_{\Gsf_{F_x}}(\bsf) = \loc_x(\kappa_\Gsf(\bsf))
   \]
  \end{subenv}
 \end{cor}
 \begin{proof}
  If $\Gsf$ is a torus, this is Lemma~\ref{lem-Kottwitz-pt-field-ext-2}. We can deduce the result for reductive $\Gsf$ by the proof of the previous proposition. If $\Gsf^{der}$ is simply connected, the claim for $\Gsf$ follows by the corollary applied to the torus $\Dsf \coloneqq \Gsf/\Gsf^{der}$. For general reductive $\Gsf$, we consider a $z$-extension $\Gsf_1 \epi \Gsf$ with $\Gsf_1^{der}$ simply connected. Then we can conclude by applying the corollary to $\Gsf_1$.
 \end{proof}
 
 \begin{cor} \label{cor-Kottwitz-pt-properties}
 Let $\bsf,\gsf \in \Gsf(\breve\Fsf)$. Then
  \begin{subenv}
   \item $\kappa_\Gsf(\sigma(\bsf)) = \sigma(\kappa_\Gsf(\bsf))$,
   \item $\kappa_\Gsf(\gsf\bsf\sigma^{-1}(\gsf))$ and $\kappa_\Gsf(\bsf)$ have the same image in $\A(\Fsf,\Gsf)$.
  \end{subenv}
 \end{cor}
 \begin{proof}
  The first assertion follows from the functoriality of $\kappa_\Gsf$, the second assertion follows from the first and the fact that $\A(\Fsf,\Gsf) = \A(\breve\Fsf,\Gsf)_{\sigma^\ZZ}$.
 \end{proof}

 \subsection{} \label{ss-Newton-point}
 We continue to consider a reductive group $\Gsf$ over $\Fsf$. We first define the Newton point $\nu_\Gsf(\Gsf)$ in the case that $\bsf \in \Gsf(\breve\Fsf)$ is a rational semisimple element. Then there exists a rational subtorus $\Tsf \subset \Gsf$  such that $\bsf \in \Tsf(\breve\Fsf)$. We define the Newton point $\nu_\Gsf(\bsf) \colon \DD_\Fsf \to \Gsf$ to be the composition of $\nu_\Tsf(\bsf)$ with the embedding $\Tsf \mono \Gsf$. Note that this construction does not depend on the choice of $\Tsf$, as $\nu$ is functorial in $\Tsf$. For general $\bsf \in \Gsf(\Fsf)$ we consider $\Nrm^{(s)}(\bsf)$ instead.

 \begin{lem*}
  There exists an $s \in \NN$ such that $\Nrm^{(s)}(\bsf)$ is semisimple.
 \end{lem*}
 \begin{proof}
  Let $s'\in \NN$ such that $\sigma^{s'}(\bsf) = \bsf$. It easily follows that $\mathrm{N}^{(s't)}(\bsf) = \left (\mathrm{N}^{(s')}(\bsf)\right )^{t}$ for any $t\in \NN$. Let $\mathrm{N}^{(s')}(\bsf) = \gamma_0\cdot u_0$ with $\gamma_0 \in \Gsf(\bar\Fsf)$ semisimple, $u_0 \in \Gsf(\bar\Fsf)$ unipotent the Jordan decomposition of $\mathrm{N}^{(s)}(\bsf) $. Since $\overline \Fsf$ is of characteristic $p$, there exists $r\geqslant0$ such that $u_0^{p^r}=1$. Thus
  \[
   \Nrm^{(s'p^r)}(\bsf) = \Nrm^{(s')}(\bsf)^{p^r} = \gamma_0^{p^r} \cdot u_0^{p^r} = \gamma_0^{p^r}
  \]
  is semisimple. 
 \end{proof}
 
 Now we choose $s \in \NN$ big enough such that the above lemma is satisfied and such that $\bsf \in \Gsf(\Fsf \otimes_{\FF_q} \FF_{q^s})$. Thus $\Nrm^{(s)}(\bsf)$ is a rational element of $\Gsf_{\Fsf \otimes \FF_{q^s}}$ and we can define
 \[
  \nu_\Gsf(\bsf) \coloneqq \nu_{\Gsf_{\Fsf \otimes \FF_{q^s}}}(\Nrm^{(s)}(\bsf)) \in \Hom_{\breve\Fsf}(\DD_\Fsf,\Gsf).
 \]
 As the $s$-th power induces an automorphism of $\DD_\Fsf$ the fraction on the right hand side is well-defined. Moreover, the definition does not depend on $s$ by Lemma~\ref{lem-Newton-pt-field-ext}. 
 
 \begin{lem} \label{lem-Newton-pt-func}
  Let $\bsf \in \Gsf(\breve\Fsf)$.
  \begin{subenv}
   \item For any (rational) morphism of reductive groups $f\colon \Gsf \to \Hsf$ we have $\nu_\Hsf(f(\bsf)) = f \circ \nu_\Gsf(\bsf)$.
   \item For any finite field extension $\Fsf'/\Fsf$, we have $\nu_{\Gsf_{\Fsf'}}(\Nrm^{(d)}(\bsf)) = \nu_\Gsf(\bsf)$, where $d$ is the degree of the induced extension on the field of constants.
   \item  For any $x\in |C|$ of degree $d$ we have
   \[
    \nu_{G_{x}}(\Nrm^{(d)}(\bsf)) = \nu_\Gsf(\bsf) \circ \iota_x.
   \]
  \end{subenv}
 \end{lem}
 \begin{proof}
  By construction, it suffices to check the above statements in the case where $\Gsf = \Tsf$ is a torus. The first statement follows directly from its definition. The last two statements are Lemma~\ref{lem-Newton-pt-field-ext}.
 \end{proof}
   
 \begin{lem} \label{lem-Newton-pt-properties}
  Let $\bsf,\gsf \in \Gsf(\breve\Fsf)$. Then
  \begin{subenv}
   \item $\nu_\Gsf(\gsf\bsf\sigma(\gsf)^{-1}) = \Int(\gsf)\circ \nu_\Gsf(\bsf)$ and
   \item $\nu_\Gsf(\sigma(\bsf)) = \sigma(\nu_\Gsf(\bsf)) = \Int(\bsf\iv) \circ \nu_G(b)$.
  \end{subenv}
 \end{lem}
 \begin{proof}
   For the first assertion, we note that we have for any $d\in\NN$ that $\Nrm^{(d)}(\gsf\bsf\sigma^{-1}(\gsf)) = \gsf\Nrm^{(d)}(\bsf)\sigma^{-d}(\gsf)$. Choosing $d$ big enough that $\sigma^{(d)}(\gsf) = \gsf$, we get by Lemma~\ref{lem-Newton-pt-func} 
  \[
   \nu_\Gsf(\gsf\bsf\sigma^{-1}(\gsf)) =  \nu_{\Gsf_{\Fsf \otimes \FF_{q^d}}}(\gsf\Nrm^{(d)}(\bsf)\gsf^{-1}) = \Int(\gsf) \circ \nu_{\Gsf_{\Fsf \otimes \FF_{q^d}}}(\Nrm^{(d)}(\bsf)) = \Int(\gsf) \circ \nu_\Gsf(\bsf).
  \]
  
 The first equation in the second statement follows from applying Lemma~\ref{lem-Newton-pt-func}~(1) to $\sigma\colon \Gsf \to \Gsf$. By the first part of the lemma we also have
 \[
  \nu_\Gsf(\sigma(\bsf)) = \nu_\Gsf(\bsf\iv \bsf \sigma(\bsf)) = \Int(\bsf) \circ \nu_\Gsf(\bsf).
 \]
  
 \end{proof}
 
 \begin{lem} \label{lem-trivial-Newton-pt}
  Let $\bsf \in \Gsf(\breve\Fsf)$. Then $\nu_\Gsf(\bsf)$ is trivial if and only if $[\bsf]$ lies in the image of $\coh{1}(\breve{\Fsf}/\Fsf,\Gsf)$ under (\ref{H1-B}).
 \end{lem}
 \begin{proof}
  We have already shown in the discussion of the morphism (\ref{H1-B}) that $[\bsf]$ lies in the image of $\coh{1}(\breve\Fsf/\Fsf,\Gsf)$ if and only if $\Nrm^{(s)}(\bsf) = 1$ for $s>0$ divisible enough (or equivalently if this holds for any $\bsf' \in [\bsf]$). Thus if $[\bsf]$ lies in the image of $\coh{1}(\breve\Fsf/\Fsf,\Gsf)$ implies that
  \[
   \nu_{\Gsf}(\bsf) = \nu_{\Gsf}(\Nrm^{(s)}(\bsf)) = 0.
  \]
  On the other hand, assume that $\nu_\Gsf(\bsf)$ is trivial. We choose $s\in \NN$ and $\Tsf \subset \Gsf_{\Fsf \otimes \FF_{q^s}}$ such that $\Nrm^{(s)}(\bsf) \in \Tsf(\breve\Fsf)$. Since $\nu_\Tsf(\Nrm^{(s)}(\bsf)) = 0$,  $\bar\kappa_\Tsf(\Nrm^{(s)}(\bsf))$ must be torsion. Thus by Lemma~\ref{lem-Kottwitz-pt-field-ext}, there exists an $s' \in \NN$ such that $\bar\kappa_{\Tsf \otimes \FF_{q^{ss'}}}(\Nrm^ {(ss')}(\bsf)) = 1$ and hence $\Nrm^ {(ss')}(\bsf) = \gsf\iv\sigma^{ss'}(\gsf)$ for some $\gsf \in \Tsf(\breve\Fsf)$. Hence $\Nrm^{(ss')}(\gsf\bsf\sigma(\gsf\iv)) = 1$.
 \end{proof} 
 
 \subsection{} \label{non-reductive case}
 Even for elements in an arbitrary linear algebraic group, the construction above still yields a well-defined Newton point, though one potentially has to enlarge $s$ such that $\Nrm^{(s)}(\bsf)$ lies in the identity component. Moreover one checks that the above lemmata generalise to linear algebraic groups, using the same proofs.
 
 As application of this, let $\Psf$ be a connected linear algebraic group whose unipotent radical $R_u\Psf$ is defined over $\Fsf$ and $\Fsf$-split (e.g.\ when $\Psf$ is a parabolic subgroup of a reductive group). Then we have $\B(\Fsf,R_u\Psf) \cong \coh{1}(\Fsf,R_u\Psf) = 1$ by Lemma~\ref{lem-trivial-Newton-pt}; thus the canonical projection $\Psf \epi \Psf_{\rm red}$ onto the reductive quotient induces an isomorphism
 $\B(\Fsf,\Psf) \cong \B(\Fsf,\Psf_{\rm red})$ by the long exact sequence in (\ref{ss-exaxt-seq}).

 \subsection{} \label{ssect-def-of-invariants} By Corollary~\ref{cor-Kottwitz-pt-properties} and Lemma~\ref{lem-Newton-pt-properties}, the Kottwitz and Newton point define invariants for every $\bbf \in \B(\Fsf,\Gsf)$:
 \begin{align*}
  \bar\kappa_\Gsf(\bbf) &\in \A(\Fsf,\Gsf) \\
  \bar\nu_\Gsf(\bbf) &\in (\Hom_{\breve\Fsf}(\DD_F,\Gsf)/\Gsf(\breve\Fsf))^{\Gal(\breve\Fsf/\Fsf)},
 \end{align*}
 which we also call the Kottwitz point and Newton point, respectively. As in (3.6.1) we can define the norm map
 \[
  \A(\Fsf,\Gsf) \to (\Vals{\scl{\Fsf}} \otimes \pi_1(\Gsf))^{\Gal(\scl{\Fsf}/\Fsf)}.
 \]
  Then the image of $\bar\kappa_\Gsf(\bbf)$ under the Norm map coincides with the image of $\bar\nu_\Gsf(\bbf)$ under the canonical projection.
  \[
  (\Hom(\DD_\Fsf,\Gsf)/\Gsf(\scl{\Fsf}))^{\Gal(\scl{\Fsf}/\Fsf)} \epi (\Div(\scl{\Fsf}) \otimes \pi_1(\Gsf))^{\Gal(\scl{\Fsf}/\Fsf)}.
 \]
  Indeed, it suffices to check this result for tori, where it is true by definition. 

 \section{Basic $\sigma$-conjugacy classes} \label{sect-basic}

 As a next step, we would like to study to which extend a $\sigma$-conjugacy class is determined by its Kottwitz and Newton point. Our starting point is the following special case.
 
  \begin{defn}
  A $\sigma$-conjugacy class $\bbf \in \B(\Fsf,\Gsf)$ is called basic, if $\bar\nu_\Gsf(\bbf)$ is central. We denote by $\B(\Fsf,\Gsf)_b \subset \B(\Fsf,\Gsf)$ the subset of basic elements. An element $\bsf \in \Gsf(\breve\Fsf)$ is called basic, if $[\bsf]$ is basic.
 \end{defn}
  
  \subsection{}  \label{ss-J_b-basic}
  By Lemma~\ref{lem-trivial-Newton-pt} we get a Cartesian diagram
  \begin{center}
   \begin{tikzcd}
    \B(\Fsf,\Gsf)_b \ar{r} \ar{d} & \B(\Fsf,\Gsf) \ar{d} \\
    \coh{1}(\Fsf,\Gsf^{\rm ad}) \ar{r} & \B(\Fsf,\Gsf^{\ad}).
   \end{tikzcd}
  \end{center}
  Hence each $\bbf \in \B(\Fsf,\Gsf)_b$ corresponds to a set  of isomorphic inner forms $\{\Jsf_\bsf\}_{\bsf\in\bbf}$ of $\Gsf$, explicitly given by $\Jsf_\bsf(\breve\Fsf) = \Gsf(\breve\Fsf)$ where the $\Gal(\breve\Fsf/\Fsf)$-action is twisted by $\Ad(\bsf)$. More explicitly, we have for every $\Fsf$-algebra $\Rsf$
  \[
   \Jsf_\bsf(\Rsf) = \{ \gsf \in \Gsf(\Rsf \otimes_\Fsf \breve\Fsf) \mid \gsf\iv \bsf \sigma(\gsf) = \bsf \}.
  \]
  
   \begin{lem} \label{lem-inner-form}
 The map $\tau_\bsf\colon \Jsf_\bsf(\breve\Fsf) \mapsto \Gsf(\breve\Fsf),\gsf \mapsto \gsf \cdot \bsf$ induces a bijection $\bar\tau_\bsf\colon \B(\Fsf,\Jsf_b) \bij \B(\Fsf,\Gsf)$ such that for all $\bsf'\in \Jsf_\bsf(\breve\Fsf) = \Gsf(\breve\Fsf)$ we have 
  \[
   \nu_\Gsf(\tau_\bsf(\bsf')) = \nu_{\Jsf_\bsf}(\bsf') + \nu_{\Gsf}(\bsf)
  \]  
 \end{lem}
 \begin{proof}
  One easily checks that the bijection $\tau_b$ preserves and reflects $\sigma$-conjugacy. We fix an unramified field extension $\Fsf' = \Fsf \otimes \FF_{q^s}$ such that there exists an isomorphism $\Jsf_{\bsf,\Fsf'} \cong \Gsf_\Fsf'$, i.e.\ $\Nrm^{(s)}(\bsf)$ is central in $\Gsf$.
   We denote by $\sigma_{\bsf} \coloneqq \Int(\bsf) \circ \sigma$ and for any $\gsf \in \Gsf(\breve\Fsf)$
  \[
   \Nrm_{\sigma_\bsf}^{(s)}(\gsf) \coloneqq \gsf \cdot \sigma_{\bsf}(\gsf) \dotsm \sigma_\bsf^{s-1}(\gsf) = \Nrm^{(s)}(\gsf\cdot \bsf) \cdot \Nrm^{(s)}(\bsf)^{-1}. 
  \] 
  Now
  \begin{align*}
   \nu_{\Jsf_\bsf}(\gsf) &= \nu_{\Jsf_{\bsf,\Fsf'}}(\Nrm_{\sigma_b}^{(s)}(\gsf))
   = \nu_{\Gsf_{\Fsf'}}(\Nrm^{(s)}(\gsf \cdot \bsf) \cdot \Nrm^{(s)}(\bsf)^{-1})
   = \nu_{\Gsf_{\Fsf'}}(\Nrm^{(s)}(\gsf \cdot \bsf)) - \nu_{\Gsf_{\Fsf'}}(\Nrm^{(s)}(\bsf)) \\
   &= \nu_\Gsf(\gsf\bsf) - \nu_\Gsf(\bsf),
  \end{align*}  
  where the third equality follows by functoriality of $\nu$ applied to the multiplication \mbox{$\Gsf \times \Cent(\Gsf) \to \Gsf$} and $(\cdot)\iv\colon \Cent(\Gsf) \to \Cent(\Gsf)$.
 \end{proof}
 
 \begin{cor} \label{cor-same-Newton-pt}
  For any $\bsf \in \B(\Fsf,\Gsf)_b$ we have a natural bijection
  \[
   \coh{1}(\Fsf,\Jsf_\bsf) \to \{\bbf' \in \B(\Fsf,\Gsf) \mid \bar\nu_\Gsf(\bbf') = \bar\nu_\Gsf(\bbf) \}
  \]
 \end{cor}
 \begin{proof}
  By Lemma~\ref{lem-trivial-Newton-pt}, $\coh{1}(\Fsf,\Jsf_\bsf)$  can be identified with $\{\bbf' \in \B(\Fsf,\Jsf_\bsf) \mid \bar\nu_{\Jsf_\bsf}(\bbf') = 0 \}$.  By the previous lemma, there is a natural bijection 
  \[
   \{\bbf' \in \B(\Fsf,\Jsf_\bsf) \mid \bar\nu_{\Jsf_\bsf}(\bbf') = 0 \} \to \{\bbf' \in \B(\Fsf,\Gsf) \mid \bar\nu_\Gsf(\bbf') = \bar\nu_\Gsf(\bbf) \},
  \]
  finishing the proof.
 \end{proof}
 
 \begin{thm}
  The Kottwitz point  induces a bijection $\bar\kappa_\Gsf\colon \B(\Fsf,\Gsf)_b \bij \A(\Fsf,\Gsf)$. 
 \end{thm} 
 \begin{proof}
  The proof in \cite[Prop.~15.1]{Kottwitz:Gisoc3} still works in our setup. However, since the proof significantly simplifies in our situation, we give the full proof for  readers' convenience.  
  
 By construction an element $\bbf \in \B(\Fsf,\Gsf)$ is basic if and only if $\bar\nu(\bbf) \circ \iota_y$ is central for all places $y$ of $\scl\Fsf$. Since $\bar\nu(\bbf)$ is rational, we have for all $y$ and $\tau \in \Gal(\scl\Fsf/\Fsf)$ that 
 $
  \bar\nu(\bbf) \circ \iota_{\tau(y)} = \tau(\bar\nu(\bbf) \circ \iota_y).
 $
 Thus it is equivalent to check only that $\bar\nu(\bbf) \circ \iota_x$ is central for all $x \in |C|$. We conclude by Lemma~\ref{lem-Newton-pt-func}~(3) that $\bbf$ is basic if and only if its image $\bbf_x \in \B(F_x,\Gsf)$ is basic for every $x \in |C|$. 
  We first prove the theorem under the assumption that $\Gsf^\der$ is simply connected. Then $\A(\Fsf,\Gsf) \cong \A(\Fsf,\Dsf)$, where $\Dsf = \Gsf/\Gsf^{\rm der}$. Thus the statement of the theorem is equivalent to the canonical map $\B(\Fsf,\Gsf)_b \to \B(\Fsf,\Dsf)$ being bijective.
  
  To prove injectivity let $\bbf,\bbf'\in \B(\Fsf,\Gsf)_b$ with identical images in $\B(\Fsf,\Dsf)$. 
   We denote by $Z(\Gsf)$ the center of $\Gsf$. Since the map $Z(\Gsf) \to \Dsf$ is an isogeny, the induced morphism $\Hom(\DD_\Fsf,Z(\Gsf)) \to \Hom(\DD_\Fsf,\Dsf)$ is injective. Hence $\nu_\Gsf(\bsf) = \nu_\Gsf(\bsf')$. By Corollary~\ref{cor-same-Newton-pt}, we have that the difference between $\bbf$ and $\bbf'$ is measured by an element $\tau \in \coh{1}(\Fsf,\Jsf_\bbf)$ in a natural way. By assumption
  \[
   \tau \in \ker(\coh{1}(\Fsf,\Jsf_\bbf) \to \coh{1}(\Fsf,\Dsf)) = \coh{1}(\Fsf,\Jsf_\bsf^{\rm der}).
  \]
  By \cite{Harder:Galkoh3} this set is trivial, hence $\bbf = \bbf'$.
  
  To prove surjectivity, we fix $\bbf'' \in \B(\Fsf,\Dsf)$ and denote by $S \subset |C|$ the (finite) set of all places $x$ where $\bbf''_x$ is non-trivial. By \cite{BuxWortman:Finiteness} there exists a maximal $\Fsf$-torus $\Tsf \subset \Gsf$ such that $\Tsf$ is elliptic over $F_x$ for all $x \in S$. By Proposition~\ref{prop-kottwitz-pt-for-tori}, it suffices to construct an element $\lambda \in \Asf(\Fsf,\Tsf)$ which maps to $\bar\kappa_\Dsf(\bbf'')$. For this we fix a Galois extension $\Esf/\Fsf$ splitting $\Tsf$ and hence $\Dsf$. Since $\bbf''_x$ is trivial outside $S$, $\bar\kappa(\bbf'')$ is the coinvariant of an element of the form
  \[
   \sum_{y \in S_\Esf} \mu_y \otimes y \in X_\ast(\Dsf) \otimes \Vals{\Esf}.
  \] 
  By \cite{Kottwitz:Gisoc1}, there exists an element $\bbf_x \in \B(F_x,\Tsf)$ whose image in $\B(F_x,\Dsf)$ equals $\bbf''_x$. The Kottwitz point of $\bbf_x$ is an coinvariant of an element of the form
  \[
   \delta_x = \sum_{y|x} \mu'_y \otimes y \in \bigoplus_{y|x} \ZZ\cdot y.
  \]
  By the combinatorial argument given in \cite[p.~80]{Kottwitz:Gisoc3}, we may choose $\delta_x$ such that $\mu'_y$ maps to $\mu_y$ for all $y$ and such that $\delta \coloneqq \sum_{x\in S} \delta_x$ is an element of $X_\ast(\Tsf) \otimes \Vals{\Esf}$. In particular the coinvariant of $\delta$ in $\A(\Fsf,\Gsf)$ satisfies the wanted property.
  
  For general $\Gsf$, we choose a $z$-extension
  \begin{center}
   \begin{tikzcd}
    1 \ar{r} & \Zsf \ar{r} & \Gsf' \ar{r} & \Gsf \ar{r}& 1
   \end{tikzcd}
  \end{center}
  such that $\Gsf'^{\rm der}$ is simply connected. Hence we have a commutative diagram
  \begin{center}
   \begin{tikzcd}
     \B(\Fsf,\Zsf) \ar{r} \ar{d}{\sim}[swap]{\bar\kappa_\Zsf} 
     & \B(\Fsf,\Gsf')_b \ar{r} \ar{d}{\sim}[swap]{\bar\kappa_{\Gsf'}}
     & \B(\Fsf,\Gsf)_b \ar{d}[swap]{\bar\kappa_\Gsf} 
     \ar{r} & 1 \\
     \A(\Fsf,\Zsf) \ar{r} & \A(\Fsf,\Gsf') \ar{r} & \A(\Fsf,\Gsf) \ar{r} & 1
   \end{tikzcd}
  \end{center}
  with exact rows (where the top row is in the category of pointed sets and exact by the long exact sequence in (\ref{ss-exaxt-seq})) and whose two left vertical arrows are isomorphisms. Thus $\bar\kappa_\Gsf$ is also an isomorphism.
 \end{proof}

 As a consequence, we obtain the following statement about the Kottwitz point.

 \begin{prop} \label{prop ad-isomorphism}
  Let $\Hsf \to \Hsf'$ be an ad-isomorphism of reductive groups. Then the diagram
  \begin{center}
   \begin{tikzcd}
    \B(\Fsf,\Hsf) \ar{r}\ar{d} & \B(\Fsf,\Hsf') \ar{d} \\
    \A(\Fsf,\Hsf) \ar{r} & \A(\Fsf,\Hsf')
   \end{tikzcd}
  \end{center}
  is Cartesian.
 \end{prop} 

 \begin{proof}
  The proof is the same as for local fields (\cite[Prop.~4.10]{Kottwitz:Gisoc2}) after one replaces $Z(H)^\Gamma$ by $\A(\Fsf,\Hsf)$. The arguments in the proof are formal cohomological constructions, which continue to hold, or statements about $\B(\Fsf,\Gsf)$, which we have proven above.
 \end{proof}

 \begin{cor} \label{cor-same-Kottwitz-pt}
  Let $\bar\kappa \in \A(\Fsf,\Gsf)$ and $\bsf \in \B(\Fsf,\Gsf)_b$ the corresponding element. Then the composition $ \B(\Fsf,\Jsf_b^{\rm sc}) \to \B(\Fsf,\Jsf_\bsf) \xrightarrow{\tau_\bsf} \B(\Fsf,\Gsf)$ induces an isomorphism $\B(\Fsf,\Jsf_\bsf^{\rm sc})\cong \B(\Fsf,\Gsf)_{\bar\kappa}$.
 \end{cor}
 \begin{proof}
  By Lemma~\ref{lem-inner-form} reduces us to show that $\B(\Fsf,\Jsf_\bsf^{\rm sc}) \to \B(\Fsf,\Jsf_\bsf)$ induces an isomorphism $\B(\Fsf,\Jsf_\bsf^{\rm sc})  \cong \B(\Fsf,\Jsf_\bsf)_0$. This holds by the previous proposition. 
 \end{proof}

 \section{Describing $\B(\Fsf,\Gsf)$ by its invariants} \label{sect-classification}
 
 \subsection{} We start with an application to the group of self-quasi-isogenies of an isoshuka; this construction will be helpful later. Given a $\Gsf$-isoshtuka $(\Vscr,\phi)$, we would like to study the group of self-quasi-isogenies $\Aut(\Vscr,\phi)$. For this we choose a trivialisation $\Vscr \cong \Gsf_{\breve\Fsf}$, which identifies $\phi$ with $\bsf\sigma$ for a $\bsf \in \Gsf(\breve\Fsf)$. We now obtain
 \[
  \Aut(\Vscr,\phi) \cong \{ \gsf \in \Gsf(\breve\Fsf) \mid \gsf\bsf = \bsf\sigma(\gsf) \} \eqqcolon \Jsf_\bsf(\Fsf).
 \] 
 \begin{prop}
  Let $\bsf \in \Gsf(\breve\Fsf)$ and $\Fsf'$ field of definition of $\nu_\Gsf(\bsf)$. Denote by $\Msf_\bsf$ the centraliser of $\nu_\Gsf(\bsf)$ in $\Gsf_{\Fsf'}$.
  \begin{subenv}
   \item $\Jsf_\bsf(\Fsf)$ is contained in $\Msf_\bsf(\breve\Fsf)$.
   \item The functor $\Jsf_\bsf\colon R \mapsto \{\gsf \in \Gsf(R\otimes_\Fsf \breve\Fsf) \mid \gsf\bsf = \bsf\sigma(\gsf)\}$ is representable by a reductive group over $\Fsf$. Moreover, $\Jsf_{\bsf,\Fsf'}$ is an inner $\breve\Fsf$-form of $\Msf_\bsf$. 
  \end{subenv}
 \end{prop} 
 \begin{proof}
  The first assertion holds as
  \[
   \Int(\gsf) \circ \nu_\Gsf(\bsf) = \nu_\Gsf(\gsf \bsf \sigma(\gsf\iv)) = \nu_\Gsf(\bsf).
  \]
  We prove the second part of (2) first. Identifying $\Fsf' \otimes \breve\Fsf \cong \breve\Fsf'^{[\Fsf':\Fsf]}$, we obtain for any $\Fsf'$-algebra $\Rsf$
  \begin{align*}
   \Jsf_\bsf(\Rsf) &= \{(\gsf_i) \in \Gsf(\Rsf \otimes_{\Fsf'} \breve\Fsf')^{[\Fsf':\Fsf]} \mid \gsf_i \bsf = \bsf \sigma(\gsf_{i-1}) \} \\
   &\cong \{\gsf_1  \in \Gsf(\Rsf \otimes_{\Fsf'} \breve\Fsf') \mid \gsf_1 \Nrm^{([\Fsf':\Fsf])}(\bsf) = \Nrm^{([\Fsf':\Fsf])}(\bsf) \sigma_{\Fsf'}(\gsf_1)\}.
  \end{align*}
  Thus $\Jsf_{\bsf,\Fsf'}$ is isomorphic to $\Jsf_{\Nrm^{([\Fsf':\Fsf])}(\bsf)}$ for the base field $\Fsf'$. Replacing $\bsf$ by $\Nrm^{([\Fsf':\Fsf])}(\bsf)$ and $\Fsf$ by $\Fsf'$, we thus reduce the second part of (2) to the case that $\nu_\Gsf(\bsf)$ is rational. Since
  \[
   \Int(\bsf) \circ \nu_\Gsf(\bsf) = \nu_\Gsf(\sigma(\bsf)) = \sigma(\nu_\Gsf(\bsf)) = \nu_\Gsf(\bsf),
  \]
  we have that $\bsf \in \Msf(\breve\Fsf)$. Since $\bsf$ is a basic element of $\Msf(\breve\Fsf)$ by construction, the claim follows from (\ref{ss-J_b-basic}). The first part of (2) now follows from the second as $\Int(\bsf) \circ \sigma$ defines a Galois descent datum on $\Jsf_{\bsf,\Fsf'}$.
 \end{proof}  
 Note that by the proof of the previous proposition, we may identify $\Jsf_{\bsf,\breve\Fsf} \cong \Msf_{\bsf,\breve\Fsf}$ where the Frobenius acts by $\Ad(\bsf) \circ \sigma$. We obtain the following intermediate results in our pursuit to classify $\Bsf(\Fsf,\Gsf)$.
 
 \begin{lem}
  Let $\bsf \in \Gsf(\breve\Fsf)$ and denote by $\nu$ its Newton point. The map $\tau_\bsf\colon \Jsf_b(\breve\Fsf) \to \Gsf(\breve\Fsf),\gsf \mapsto \gsf\cdot \bsf$ induces a bijection 
  \[  
   \coh{1}(\Fsf,\Jsf_\bsf) \bij \{ \bbf' \in \B(\Fsf,\Gsf) \mid \bar\nu_\Gsf(\bbf') = \bar\nu \}.
  \]
 \end{lem}
 \begin{proof}
  By the proof of Lemma~\ref{lem-inner-form} the map $\tau_\bsf$ preserves $\sigma$-conjugacy and we get for any $\jsf \in \Jsf_\bsf(\breve\Fsf)$ that $\nu_\Gsf(\tau_\bsf(\jsf)) = \nu_{\Jsf_\bsf}(\jsf) + \nu_\Gsf(\bsf)$. Hence the above map on $\sigma$-conjugacy classes is well-defined. 
  
  To prove injectivity, let $\jsf,\jsf' \in \Jsf_\bsf(\breve\Fsf)$ such that $\nu_{\Jsf_\bsf}(\jsf) = \nu_{\Jsf_\bsf}(\jsf') = 0$ and $\gsf \in \Gsf(\breve\Fsf)\iv$ such that $\tau_\bsf(\jsf') = \gsf \tau_\bsf(\jsf) \sigma(\gsf)$. By Lemma~\ref{lem-Newton-pt-func} $\gsf$ centralises $\nu_{\Gsf}(\tau_\bsf(\jsf)) =  \nu_{\Gsf}(\tau_\bsf(\jsf')) = \nu$. Hence $\gsf \in \Msf_\bsf(\breve\Fsf) = \Jsf_\bsf(\breve\Fsf)$ and one concludes $\jsf' = \gsf \jsf \bsf \sigma(\gsf)\iv \bsf\iv= \gsf \jsf \sigma_\bsf(\gsf)\iv$.
  
 For surjectivity let $\bbf' \in \B(\Fsf,\Gsf)$ with $\bar\nu_\Gsf(\bbf) = \bar\nu$ and choose $\bsf' \in \bbf'$ with $  \nu_\Gsf(\bsf')= \nu$. Substituting $\gsf = \bsf^{-1},\bsf'^{-1}$in Lemma~\ref{lem-Newton-pt-properties}~(2) we obtain $\sigma(\nu) = \Int(\bsf^{-1}) \circ \nu = \Int(\bsf'^{-1}) \circ \nu$. In particular, we have $\jsf \coloneqq \bsf' \bsf^{-1} \in \Msf_\bsf(\breve\Fsf) = \Jsf_\bsf(\breve\Fsf)$. In particular,  $   \nu_{\Jsf_\bsf}(\jsf) = \nu_\Gsf(\bsf') - \nu_{\Gsf}(\bsf)
  $
  is trivial, proving surjectivity.
 \end{proof}
 
 \begin{rmk}
  One checks that analogously to the result of Rapoport and Richartz \cite[Prop.~1.17]{RapoportRichartz:Gisoc}, the $\Jsf_\bsf$-torsor corresponding to $\bsf'$ as above is given by
  \[
   \Jsf_{\bsf,\bsf'}\colon R \mapsto \{ \gsf \in \Gsf(R \otimes \Fsf) \mid \gsf \bsf \sigma(\gsf)^{-1} = \bsf' \}.
  \]
 \end{rmk}
 
  \begin{cor}
 Let $\bbf,\bbf' \in \B(\Fsf,\Gsf)$ such that $\bar\nu_\Gsf(\bbf) = \bar\nu_\Gsf(\bbf')$ and $\bar\kappa_\Gsf(\bbf) = \bar\kappa_\Gsf(\bbf')$. Then $\bbf = \bbf'$.
 \end{cor}
 \begin{proof}
  By the previous corollary, we may replace $\Gsf$ by $\Jsf_{\bbf'}$ and assume that $\bbf' = 1$. By the Corollary~\ref{cor-same-Kottwitz-pt}, we may assume that $\Gsf$ is simply connected. By Lemma~\ref{lem-trivial-Newton-pt}, $\bbf$ corresponds to an element in $\coh{1}(\Fsf,\Gsf)$, which is trivial because $\Gsf$ is simply connected (by \cite{Harder:Galkoh3}, or alternatively $\coh{1}(\Fsf,\Gsf) \mono \A(\Fsf,\Gsf) = \{0\})$).
 \end{proof}
 
 \subsection{} By the previous corollary, we know that the Kottwitz and Newton point uniquely determine a $\sigma$-conjugacy class. If $\Gsf$ is quasi-split, the possible values of $(\bar\nu,\bar\kappa)$ can be calculated using Theorem~\ref{thm-steinberg} below; so we briefly describe how to reduce the description of $\Bsf(\Fsf,\Gsf)$ to the quasi-split case.
 
  By Proposition~\ref{prop ad-isomorphism} we can reduce to the case that $\Gsf$ is of adjoint type. Then any inner form of $\Gsf$ is can be described by twisting the Frobenius action with a basic element, thus Lemma~\ref{lem-inner-form} allows us to replace $\Gsf$ by a quasi-split inner form in this regard.

 \begin{thm} \label{thm-steinberg}
  Assume that $G$ is quasi-split. Then the map
  \[
   \bigsqcup_{\Tsf \subset \Gsf \atop max. \Fsf-torus} \B(\Fsf,\Tsf) \to \B(\Fsf,\Gsf)
  \]
  is surjective.
 \end{thm}   
 \begin{proof}
  We first assume that the derived group of $\Gsf$ is simply connected. Let $\bsf \in \Gsf(\breve\Fsf)$ be arbitrary. We fix an element $\zsf \in \Jsf_\bsf(\Fsf)$ which is regular semisimple in $\Gsf$. Such an element exists, since the set of $\Gsf$-regular semisimple elements in $\Jsf_\bsf(\cl{\Fsf})$ is open and non-empty as $\Jsf_b$ contains a maximal torus of $\Gsf$; as $\Jsf_\bsf(\Fsf)$ is a dense subset of $\Jsf_\bsf(\cl{\Fsf})$, it must intersect this set non-emptily.
  Now
  $
   \sigma(\zsf) = \bsf\iv \zsf \bsf,
  $
  in particular the conjugacy class $C$ of $\zsf$ in $\Gsf(\breve\Fsf)$ is rational. By Corollary~\ref{cor} there exists an $\Fsf$-rational element $\zsf_0 \in C$. Write $\zsf = \gsf \zsf_0 \gsf\iv$. Then we get
 \[
  \sigma(\gsf) \sigma(\zsf_0) \sigma(\gsf)\iv  = \sigma(\zsf) = \bsf\iv \gsf \zsf_0 \gsf\iv \bsf.
 \]
 Hence $\gsf\iv \bsf\sigma(\gsf)$ is an element of the centraliser of $\zsf_0$, which is an $\Fsf$-torus by assumption.
 
 To prove the remaining cases, we consider a $z$-extension
  \begin{center}
   \begin{tikzcd}
    1 \ar{r} & \Zsf \ar{r} & \Gsf_1 \ar{r} & \Gsf \ar{r} & 1
   \end{tikzcd}
  \end{center}
 with $\Gsf_1^{\rm der}$ is simply connected.  We obtain a commutative diagram
 \begin{center}
  \begin{tikzcd}
   \bigsqcup\limits_{\Tsf_1 \subset \Gsf_1 \atop max. \Fsf-torus} \B(\Fsf,\Tsf_1) \ar[two heads]{d} \ar[two heads]{r} & \B(\Fsf,\Gsf_1) \ar[two heads]{d}\\
   \bigsqcup\limits_{\Tsf \subset \Gsf \atop max. \Fsf-torus} \B(\Fsf,\Tsf) \ar{r} & \B(\Fsf,\Gsf).
  \end{tikzcd}
 \end{center}
 Hence the lower map must also be surjective.
 \end{proof}
 
 For future reference, we record the following helpful result.
 
 \begin{cor}
  Let $\Gsf$ be quasi-split and $\bbf \in \B(\Fsf,\Gsf)$. Then there exists $\bsf \in \bbf$ such that $\nu_\Gsf(\bsf)$ is rational.
 \end{cor} 
 \begin{proof}
  This follows from the previous theorem since the Newton point of a torus is automatically rational.
 \end{proof}

\appendix
\section{On rational conjugacy classes}
 \subsection{Statement}
 Let $F$ be \emph{any} field. (In the intended setting, $F$ will be a global or local function field, which is \emph{imperfect}.) Let $G$ be a (connected) \emph{quasi-split} reductive group over $F$, and assume that the derived subgroup $G^{\der}$ of $G$ is \emph{simply connected}. 
 
Note that any conjugacy class  $C\subset G(\cl{F})$ is the set of $\cl F$-points of some locally closed subvariety of $G_{\cl F}$. We say that a conjugacy class $C\subset G(\cl{F})$ is defined over $F$ if the corresponding locally closed subvariety is defined over $F$.

 \begin{thmsub}\label{thm1}
Let $C\subset G(\cl{F})$ be a \emph{regular semisimple} conjugacy class defined over $F$. Then there exists an element $x\in G(F)\cap C$.
\end{thmsub}
If $F$ is perfect then the above theorem was obtained for any $F$-rational semisimple conjugacy class $C$  (but not necessarily regular); \emph{cf.} \cite[Thm.~4.1]{Kottwitz:RatlConjRedGp}. If $F$ is not necessarily perfect and $G$ is  semisimple, then the above theorem was proved in \cite[\S~8.6]{BorelSpringer:RationalityPropertiesII} by modifying the argument for perfect base fields in \cite[Thm.~1.7]{Steinberg:RegularElements}. In the semisimple case, the proof uses the \emph{regularity assumption} on $C$ if the base field is imperfect. 

In short, the above proposition can be obtained by modifying the proof of Steinberg's theorem \cite[Thm.~1.7]{Steinberg:RegularElements} in the reductive setting, using some ideas from \cite[Thm.~4.1]{Kottwitz:RatlConjRedGp}. As we are not aware of any reference, let us provide a proof here.

We also need the following refinements of Theorem~\ref{thm1}:

\begin{corsub}\label{thm2}
We keep all the assumptions from Theorem~\ref{thm1}, and \emph{assume} that $F$ has cohomological dimension $\leqslant 1$. Then any regular semisimple conjugacy class $C\subset G(\cl{F})$ defined over $F$ contains a \emph{unique} $G(F)$-conjugacy class.
\end{corsub}
When $F$ is perfect, this theorem was proved by Steinberg \cite[Corollary~10.3]{Steinberg:RegularElements} for any semisimple conjugacy classes.
\begin{proof}
 We first show that $C$ contains a unique $G(\scl{F})$-conjugacy class $\scl{C} \subset G(\scl{F})$, necessarily defined over $F$. Existence is guaranteed by Theorem~\ref{thm1}. On the other hand let $x,y \in C \cap G(\scl{F})$. Since all maximal tori are conjugated to each other over $\scl{F}$, we can replace $x$ and $y$ by $G(\scl{F})$-conjugates so that they are contained in a given $\scl{F}$-torus $T$. Since $x$ and $y$ are conjugated under $G(\cl{F})$, they lie in the same Weyl group orbit. Since every Weyl group element lifts to an element of $G(\scl{F})$ by \cite[Exp.~XXII, Cor.~3.8]{SGA3.3}, it follows that $x$ and $y$ are $G(\scl{F})$-conjugates of each other.

The statement now follows from the argument in \cite[\S~3]{Kottwitz:RatlConjRedGp}. Let $x\in C\cap G(F)$, which exists by Theorem~\ref{thm1}. Using the assumption that $G^{\der}$ is simply connected, it follows that  the set of $G(F)$-conjugacy classes contained in $\scl{C}$ is in natural bijection with
	\[\ker \left(\mathrm{H}^1(F,G_x) \to \mathrm{H}^1(F,G) \right),\]
	where $G_x$ is the centraliser of $x$ in $G$; \emph{cf.} \cite[\S~3]{Kottwitz:RatlConjRedGp}. On the other hand, if $F$ has cohomological dimension $\leqslant1$ then $\mathrm{H}^1(F,G_x)$ vanishes as $G_x$ is connected reductive; \emph{cf.}~\cite[\S\S~8.6]{BorelSpringer:RationalityPropertiesII}.
\end{proof}

\begin{corsub}\label{cor}
	Let $G$ be as in Theorem~\ref{thm1}, and let $E/F$ be a (not necessarily finite) Galois extension. Assume that $E$ has cohomological dimension $\leqslant 1$. Then any regular semisimple $\Gal(E/F)$-stable conjugacy class in $G(E)$ contains an element in $G(F)$.	
	
	In particular, if $F$ is a global function field and $E\coloneqq\breve F $ then any $\sigma$-stable regular semisimple conjugacy class in $G(\breve F)$ contains an element in $G(F)$.
\end{corsub}
\begin{proof}
 This is an immediate consequence of the previous corollary.
\end{proof}

It remains to prove Theorem~\ref{thm1}. Unfortunately, we cannot naively adapt the proof of \cite[Theorem~4.1]{Kottwitz:RatlConjRedGp}. Indeed, to reduce the proof to the semisimple simply connected case, \emph{loc.~cit.} used the equality
\[G(\cl{F}) = Z(\cl{F})\cdot G^{\der}(\cl{F}),\]
where $Z$ is the centre of $G$. On the other hand, we do not have such an equality over $\scl{F}$ as one can see from the case when $G=\GL_n$ with $p|n$. Therefore, we have to redo part of proof of Steinberg's theorem in \cite[\S9]{Steinberg:RegularElements} in the reductive setting. Let us now give a proof.
 \subsection*{Notations}
We choose a maximal split $F$-torus of $G$, and let $T\subset G$ be its centraliser. (Then $T$ is a maximal $F$-torus because $G$ is quasi-split.) We choose a $F$-rational Borel subgroup $B\subset G$ containing $T$. We work with the (absolute) root datum with respect to  $T_{\scl{F}}$, and we get a natural $\Gal(\scl{F}/F)$-action on roots, coroots, and the (absolute) Weyl group $W$. Therefore, we may view $W$ as an finite \'etale group scheme over $F$, and its action on $T$ is defined over $F$.

We set $T^{\der}\coloneqq T\cap G^{\der}$ and let $Z$ denote the centre of $G$. Let $D\coloneqq G/G^{\der}$ denote the cocentre of $G$. Then we have the following short exact sequence of tori:
\begin{equation}\label{eq:ses-tori}
\xymatrix@1{1\ar[r] & T^{\der} \ar[r] &T \ar[r] & D\ar[r] &1}.
\end{equation}

\begin{lemsub}\label{lem:tori}
Assume that $G^{\der}$ is simply connected.
\begin{enumerate}
\item\label{lem:tori:res} The torus $T^{\der}$ is a product of $\Res_{F_i/F}\GG_m$ for finite separable extensions $F_i/F$.
\item\label{lem:tori:ses} For any algebraic separable extension $F'/F$, the short exact sequence~\eqref{eq:ses-tori} induces a short exact sequences on $F'$-points ; i.e., 
\[ \xymatrix@1{1\ar[r] & T^{\der}(F') \ar[r] &T(F') \ar[r] & D(F')\ar[r] &1}.\]
\end{enumerate}
\end{lemsub}
\begin{proof}
Since $G^{\der}$ is simply connected, it is a product of $F$-simple factors. So the argument in \cite[p.~793]{Kottwitz:RatlConjRedGp} shows claim~\eqref{lem:tori:res}. Then claim~\eqref{lem:tori:ses} when $T^{\der}$ splits over $F'$ follows from the vanishing of $\mathrm{H}^1_{\rm fl}(F',\GG_m)$ (\emph{cf.} \cite[Ch.~III, Proposition~4.9]{milne:etale}), and the general case of \eqref{lem:tori:ses} follows from \eqref{lem:tori:res} via Shapiro's lemma and Hilbert's~Theorem~90 for Galois cohomology. 
\end{proof}

\subsection{Normal forms: Review of Steinberg's construction}
We review the work of Steinberg's on `normal forms' for regular conjugacy classes over an algebraically closed field $\overline F$. We also investigate which results hold over a separably closed fields $\scl{F}$.

In this section, let us assume that $G$ is \emph{semi-simple}, simply connected and quasi-split over $F$. (Later, we will apply the constructions in this subsection to the derived subgroup.)

Let $r$ be the absolute rank of $G$, and choose (absolute) simple roots $\alpha_1,\cdots,\alpha_r$ corresponding to our choice of $(T,B)$. Let $\sigma_i$ denote the reflection with respect to $\alpha_i$, which is an element of the absolute Weyl group, and we choose a lift $\dot\sigma_i\in N_{G}(T)(\scl{F})$ of $\sigma_i$.  Let $X_i$ denote the root subgroup of $\alpha_i$, which is a $1$-dimensional unipotent subgroup of $G$ defined over $\scl{F}$. Note that the absolute Galois group $\Gal(\scl{F}/F)$ acts on $\{\alpha_1,\cdots,\alpha_r\}$, hence it permutes $\sigma_i$'s and $X_i$'s.
 
Let us consider the following subvariety\footnote{Note that $N$ is \emph{a priori} a constructible subset, though it turns out to be a closed subvariety.}  of $G_{\scl{F}}$ (not just over $\cl{F}$):
\begin{equation}
N \coloneqq \prod_{i=1}^r (X_i \cdot \dot\sigma_i).
\end{equation} 
Steinberg \cite{Steinberg:RegularElements} defined $N$ only over $\cl F$, which is one of the reasons why he needed to assume $F$ is perfect. Then Borel and Springer \cite[\S8]{BorelSpringer:RationalityPropertiesII} realised that under the stronger foundation in algebraic groups many things in \cite{Steinberg:RegularElements} can be done over separably closed fields in place of algebraically closed fields (except Jordan decomposition). In particular, $N$ can be defined over $\scl{F}$.

\begin{propsub}\label{prop:Steinberg} Let $G$ be a semi-simple, simply connected and quasi-split reductive group over $F$.
\begin{enumerate}
\item (\emph{Cf.} \cite[Theorem~7.1]{Steinberg:RegularElements}) The natural map $\prod_{i=1}^r X_i \mapsto N$ is an isomorphism of schemes over $\scl{F}$. Furthermore, $N$ is \emph{closed} in $G$.
\item \label{prop:Steinberg:indep} (\emph{Cf.} \cite[Lemma~7.5]{Steinberg:RegularElements}) Let us consider a different lift $\dot\sigma_i'$ of $\sigma_i$ for each $i$, and set $N'\coloneqq\prod_{i=1}^r(X_i\dot\sigma_i')$. Then there exists $t,t'\in T(\scl{F})$ such that 
\[N' = t'\cdot N = tNt^{-1}.\]
\item \label {prop:Steinberg:perm} (\emph{Cf.} \cite[Proposition~7.8]{Steinberg:RegularElements}) For any permutation $\tau$ of $\{1,\cdots,r\}$ we set
\[
N^\tau \coloneqq \prod_{i=1}^r X_{\tau(i)}\dot\sigma_{\tau(i)}.
\]
Then each element of $N^\tau (\scl{F}) $ is $G(\scl{F})$-conjugate to some element in $N (\scl{F})$.
\end{enumerate}
	
\end{propsub}

We will show now that if $G$ does \emph{not} have any odd special unitary factor it is possible to choose $N$ so that it is defined over the same ground field $F$ as $G$.

Let $E$ be a finite Galois extension of $F$ over which  $G$ splits. Clearly the root subgroups $X_i$ are defined over $E$. Furthermore, by a simple application of Hilbert's Theorem~90 for split tori it is possible to take $\dot\sigma_i\in G(E)$; \emph{cf.} \cite[Lemma~9.3]{Steinberg:RegularElements}.

Now, we pick a representative $\alpha_i$ in each $\Gal(E/F)$-orbit of simple roots $\{\alpha_1,\cdots,\alpha_r\}$ and pick $\dot\sigma_i\in G(E)$. For $\alpha_j = \alpha_i^\gamma$ for $\gamma\in\Gal(E/F)$, we pick $\sigma_j\coloneqq\sigma_i^\gamma$. We also reorder $(\alpha_i)$ so that $\Gal(E/F)$-conjugates have adjacent indices.  We construct $N$ using this choice.
\begin{thmsub}[\emph{cf.~}{\cite[Theorems~9.2,~9.4]{Steinberg:RegularElements}}]\label{thm:NormalForms}
We maintain the assumption that $G$ is semi-simple, simply connected and quasi-split over $F$. Assume furthermore that $G$ does \emph{not} have an odd special unitary factor.  Then  the closed subvariety $N\subset G_{\scl{F}}$ constructed under the choices made as above is defined over $F$. 
\end{thmsub}
\begin{proof}
The case when $G$ is split over $F$, the theorem is obvious  as $X_i$ and $\dot\sigma_i$ are all defined over $F$; \emph{cf.} \cite[Theorem~9.2]{Steinberg:RegularElements}. If $G$ is quasi-split over $F$ and split over $E/F$, then clearly $N$ (constructed under the choice made as above) is naturally defined over $E$. It remains to show that $N$ is stable under the $\Gal(E/F)$-action. 

Since $\Gal(E/F)$ permutes $X_i\cdot\dot\sigma_i$, the desired Galois stability boils down to the property that 
\[(X_i\cdot\dot\sigma_i)\cdot (X_i^\gamma\cdot\dot\sigma_i^\gamma) = (X_i^\gamma\cdot\dot\sigma_i^\gamma)\cdot (X_i\cdot\dot\sigma_i), \quad\forall i,\ \forall \gamma\in\Gal(E/F)\]
By inspecting all the possible automorphisms of Dynkin diagrams, the above property holds except for the non-trivial diagram automorphism of $A_{2n}$, which corresponds to odd special unitary groups; \emph{cf.} proof of Theorem~9.4 in \cite{Steinberg:RegularElements}.
\end{proof}
\begin{exasub}
	Let us give an example of quasi-split $\SU_4$. Let $G\coloneqq\SU_4$ attached to a rank-$4$ totally isotropic hermitian space for $E/F$. Then $G_E\cong \SL_4$, and we may make a standard choice of simple roots $\alpha_1,\alpha_2,\alpha_3$. Let $\bar{(\bullet)}$ denote the unique non-trivial element in $\Gal(E/F)$. Then we have $\bar\alpha_1 = \alpha_3$ and $\bar\alpha_2=\alpha_2$. We choose $\dot\sigma_i$ so that $\dot\sigma_2$ is defined over $F$ and $\bar{\dot\sigma}_1 = \dot\sigma_3$.
	
	Let us set $N =  (X_2\dot\sigma_2)(X_1\dot\sigma_1)(X_3\dot\sigma_3)$. (Mind the order of the indices.) Then, since $\alpha_1$ and $\alpha_3$ are orthogonal, it follows that $X_1\dot\sigma_1$ and $X_3\dot\sigma_3$ commute so we have $N = \bar N$.
	
	Let us now explain why odd special unitary factors should be excluded in the above theorem. For simplicity, let us consider quasi-split $\SU_3$. Over some separable quadratic extension $E/F$ we have two simple roots $\alpha_1$ and $\alpha_2$, which are permuted by the non-trivial element of $\Gal(E/F)$. So from our choice of $\dot\sigma_i$, we have $N = (X_1\dot\sigma_1)(X_2\dot\sigma_2)$ and $\bar N = (X_2\dot\sigma_2)(X_1\dot\sigma_1)$. It does not seem possible to choose $\dot\sigma_1$ so that $N=\bar N$.
\end{exasub}

\subsection{Normal forms: Case without odd special unitary factor} The statements up until Proposition~\ref{prop:class-fcn} work for an arbitrary reductive group $G$, we only need the additional requirements in the last step. Assume that $E$ is a field extension of $F$ over which $G$ is split. (For example, we will later focus on the case when $E = \scl{F}$ or $E=\cl F$.) We denote by $G \sslash G$ the GIT quotient of $G$ with respect to the conjugation, and by $T\sslash W$ the GIT quotient of $T$ modulo its Weyl group in $G$.

\begin{lemsub}[{cf.~\cite[Cor.~6.4]{Steinberg:RegularElements}}]
 The canonical morphism $T\sslash W \to G \sslash G$ is an isomorphism.
\end{lemsub}
\begin{proof}
 Since both quotients are uniform, it suffices to consider to prove the claim over $E$. Here the proof is the same as Steinberg's proof for semi-simple $G$ if one replaces 7.15 of \emph{loc.~cit.} by \cite[Prop.~II.2.4]{Jantzen:RepAlgGr}.
\end{proof}

 Thus we have the following natural maps:
\begin{equation}\label{eq:chi-red}
	\underline\chi\colon
	\xymatrix@1{G \ar[r]& (G\sslash G) \ar[r]^-{\cong} & (T\sslash W) },
\end{equation} 
 
Note that $\underline\chi$ is constant on each conjugacy class, being conjugation-invariant. More precisely, we obtain the following result
\begin{propsub}\label{prop:class-fcn}
The fibre of $\underline\chi$ at any point $\underline c\in (T\sslash W)(\cl F)$ consists of finitely many conjugacy classes with the semisimple part determined by $\underline c$. Furthermore, $\underline\chi^{-1}(\underline c)$ contains
\begin{subenv}
	\item a unique regular conjugacy classes in $G(\cl F)$, and
	\item a unique semisimple conjugacy classes over $G(\cl F) $.
\end{subenv}
Furthermore, $\chi^{-1}(\underline c)$ consists of a single conjugacy class $C$ if and only if it is a regular semisimple conjugacy class. In this case we have $\underline c$ is an $F$-rational point of $(T\sslash W)$ if and only if $C$ is defined over $F$.
\end{propsub}
\begin{proof}
The rationality statement follows from the rest of the proposition, as it states that $C$ is the preimage of $\underline c$ and $\underline c$ is the scheme-theoretic image of $C$, respectively. Since we have $G(\cl F) = Z(\cl F)\cdot G^{\der}(\cl F)$, any conjugacy class in $G(\cl F)$ is a translate of a conjugacy class in $G^{\der}(\cl F)$ by some element $z\in Z(\cl F)$. Therefore, to prove the proposition it suffices to handle the case when $G$ is semisimple (and simply connected), which is proved in \cite[Corollary~6.6]{Steinberg:RegularElements}.
\end{proof}

We note that the canonical projection $T\twoheadrightarrow D$ is $W$-equivariant and thus factors through $T \sslash W$. Let $(T\sslash W)_{\bar t}$ denote the preimage of $\bar t\in D(F)$. The following result is essentially due to Steinberg:

\begin{thmsub}[Steinberg]\label{thm:Steinberg} Assume that $G^{\rm der}$ is simply connected and quasi-split over $F$ with no odd special unitary factor, and let $N$ be a closed $F$-subvariety of $G$ defined as in Theorem~\ref{thm:NormalForms}. Let $\bar t\in D(F)$, and choose a lift $t\in T(F)$ of $\bar t$ by Lemma~\ref{lem:tori}(\ref{lem:tori:ses}). Then $\underline\chi$ induces an isomorphism
\[\underline\chi\colon\xymatrix@1{t\cdot N \ar[r]^-{\cong}& (T\sslash W)_{\bar t}}.\]
\end{thmsub}
\begin{proof}
By flat descent, it suffices to prove the claim over $\cl F$. Let $z\in Z(\cl F)$ be a lift of $\bar t$,  so we have $z^{-1}t\in T^{\der}(\cl F)$.
By Proposition~\ref{prop:Steinberg}\eqref{prop:Steinberg:indep}, there exists $s\in T^{\der}(\cl F)$ such that $s^{-1} N_{\cl F} s=(z^{-1}t)\cdot N_{\cl F}$. Finally, 
the following isomorphism over $\cl F$  can be obtained from \cite[Corollary~7.16]{Steinberg:RegularElements}
\[
\underline\chi\colon\xymatrix@1{t\cdot N_{\cl F}  \ar[r]^-{\Int(s)}&
z\cdot N_{\cl F} \ar[r]^-{\cong} &
z\cdot (T^{\der}\sslash W)_{\cl F} \cong (T\sslash W)_{\bar t,\cl F}},
\]
which finishes the proof. 
 \end{proof}
 
 \begin{proof}[Proof of Theorem~\ref{thm1} in the case that $G^{\rm der}$ has no odd special unitary factor:] Let $C$ be a regular semisimple conjugacy class defined over $F$ and let $\underline c \in (T\sslash W)(F)$ be its image under $\underline \chi$.  Denote by $\tbar \in D(F)$ its image under the canonical projection $G \epi D$ and fix a lift $t \in T(F)$, which exists by  Lemma~\ref{lem:tori}(\ref{lem:tori:ses}). By the previous theorem there exists a unique $g \in t\cdot N(F)$ such that $\underline\chi (g) = \underline c$. By Proposition~\ref{prop:class-fcn}, this implies that $g$ is contained in $C$
\end{proof}

\subsection{Normal forms: Odd special unitary case}
In the presence of odd special unitary factor in $G^{\der}$, it is unclear if there exists a closed subvariety $N\subset G^{\der}$ defined over $F$ such that $\underline\chi$ \eqref{eq:chi-red} sends $N$ isomorphically onto $(T^{\der}\sslash W)$. Therefore, we will find another closed subvariety $N'\subset G^{\der} $ defined over $F$ such that $\underline{\chi}$ restricted to $N'$ induces a bijection on closed points and then repeat the proof given in the previous section.

We choose a finite separable extension $F'/F$ and a separable quadratic extension $E'/F'$. Let $(\bar\bullet)\in\Gal(E'/F')$ denote the non-trivial element.

We choose a basis $V = E'^{2n+1}$ and the hermitian form given by the anti-diagonal matrix with entry $1$. Let $w$ be such a matrix. By abuse of notation we let  $\SU_{2n+1}/F'$  denote the special unitary group associated to this hermitian space.  It is a semi-simple, simply connected and quasi-split  group defined over $F'$. More explicitly, we have
\[
\SU_{2n+1}(R) \coloneqq \left \{ g\in \SL_{2n+1}(E'\otimes_{F'}R)\ |\ ^{t}\bar g w g = w \right \}
\]
The defining condition can be rewritten as $g = w\cdot {}^{t}\bar g^{-1} \cdot w$. (Note that $w = w^{-1}$.) Therefore, we can identify $\SU_{2n+1}(F')$ as the $\Gal(E'/F')$-invariance of $\SL_{2n+1}(E')$ with respect to the action of $\sigma$ twisted as follows: 
\begin{equation}\label{eq:cocycle}
g\mapsto w\cdot {}^{t}\bar g^{-1} \cdot w. 
\end{equation}

It turns out that any quasi-split special unitary group can be written in this way. (Although there is another maximally isotropic hermitian $E'$-space with rank~$2n+1)$ that is not isomorphic to $V$, the associated special unitary group is isomorphic to the one associated to $V$.)

A maximal split torus of $\SU_{2n+1}/F'$ is isomorphic to $\GG_m^{n}$ where $(t_1,\cdots, t_n)\in F'^{\times}$ gets  mapped to the diagonal matrix $\diag(t_1,\cdots, t_n,1,t_n^{-1},\cdots,t_1^{-1})\in\SU_{2n+1}(F')$. Its centraliser, which is a maximal torus denoted as $T$, is isomorphic to $\Res_{E'/F'}\GG_m^n$ where $(t_1,\cdots, t_n)\in E'^{\times}$ gets mapped to gets  mapped to $\diag(t_1,\cdots, t_n,\prod_i (\bar t_i/t_i),\bar t_n^{-1},\cdots,\bar t_1^{-1})\in T(F')$. We can also fix an $F$-rational Borel subgroup $B\subset \SU_{2n+1}$ to be the ``upper-triangular matrices'', which contains $T$.

Using the above notation, let $G' \coloneqq \SU_{2n+1}/F'$ be the special unitary group associated to a rank-$(2n+1)$ maximally isotropic hermitian space over $E'$. (So $G'$ is defined over $F'$ and its adjoint quotient is absolutely simple.)

We work with the `standard' root datum for $G'_{E'} \cong \SL_{2n+1}/ E'$, and obtain $(\alpha_i)_{i=1,\cdots 2n}$ and $\sigma_i\in W$. We twist the action of $\Gal(E'/F')$  on $\SL_{2n+1}/E'$ as in \eqref{eq:cocycle}, so that the non-trivial element of $\Gal(E'/F')$ acts on the absolute root datum and the Weyl groups as follows:
\begin{align*}
\alpha_i & \longmapsto  \alpha_{2n+1-i}\\
\sigma_i & \longmapsto \sigma_{2n+1-i}
\end{align*}
In particular, $\alpha_n$ and $\alpha_{n+1}$ are Galois-conjugates and they are \emph{not} orthogonal, which is  why the subvariety  of rational normal forms $N = \prod_{i=1}^{2n} (X_i\cdot\dot\sigma_i)$ is not defined over $F'$. Indeed, the only problematic factors are $(X_n\dot\sigma_n)\cdot(X_{n+1}\dot\sigma_{n+1})$, so we modify the definition of $N$ as follows. (We faithfully follow the proof of \cite[Theorem~9.7]{Steinberg:RegularElements}.)
 
Let us first stratify $V\coloneqq T\sslash W$ as follows. Note that $ (T\sslash W)_{\scl{E'}} \cong \AA^{2n}$ where the isomorphism is given by the coefficients of the characteristic polynomial. We can stratify $ (T\sslash W)_{\scl{E'}}$ as follows: $V_{1,E'}\subset V_{E'}$ is an affine hyperplane defined by the condition that $1$ is an eigenvalue, and $V_{0,E'}\coloneqq V_{E'}\setminus V_{1,E'}$. Since $V_{0,E'}$ and $V_{1,E'}$ are $\Gal(E'/F')$-equivariant, they descend to $F'$-subvariaties $V_0$ and $V_1$.

The strategy now is to construct disjoint closed $F'$-subvarieties $N_0', N_1'\subset G'\coloneqq\SU_{2n+1}/F'$ such that the natural map 
\[\underline\chi:G'\cong \SU_{2n+1}/F' \to (T\sslash W) \qquad\text{\eqref{eq:chi-red}}\] induces the following isomorphisms:
\begin{align}
\label{eq:chi1} &\xymatrix@1{N'_{1} \ar[r]^-{\cong}& V_1}\\
\label{eq:chi0} &\xymatrix@1{N'_{0} \ar[r]^-{\cong}& V_0}.
\end{align}

This was done in \cite[Theorem~9.7]{Steinberg:RegularElements} at least when $F$ is perfect. Let us recall the construction and basic properties of $N_0'$ and $N_1'$.
\begin{defnsub}
Let us introduce some notation first.
\begin{enumerate}
	\item Let $\alpha\coloneqq\alpha_n+\alpha_{n+1}$, which is a Galois-invariant root. We let $X_{\alpha}$ and $X_{-\alpha}$ respectively denote the root subgroup of $\alpha$ and $-\alpha$. 
	
	Let $G_\alpha$ be the subgroup generated by $X_{\alpha}$ and $X_{-\alpha}$, which is isomorphic to $\SL_2$. Let $T_\alpha\coloneqq T \cap G_\alpha$ denote the maximal torus of $G_\alpha$.

Let $\sigma_\alpha\in W$ be the reflection with respect to $\alpha$, and choose its lift $\dot\sigma_\alpha$ in $G_\alpha(F')$.
	\item As before, let $X_i$ denote the root subgroup of $\alpha_i$. Let $G_i$ denote the ``$\SL_2$-subgroup'' associated to $\alpha_i$. Note that  $G_i$ can only be defined over $E'$. 
	\item 
	For any $i=1,\cdots n-1$ we choose $\dot\sigma_i$ in $G_i(E')$ and set $\dot\sigma_{2n+1-i}\coloneqq\bar{\dot\sigma}_i$, which clearly lifts $\sigma_{2n+1-i}$ and lies in $G_{2n+1-i}(E')$.
	\item Let us choose $u_n\in X_n(E')\setminus\{1\}$ and $u_{n+1}\in X_{n+1}(E')\setminus\{1\}$. 
\end{enumerate}

Under these choices, we define $N_0'$ and $N_1'$ as follows:
\begin{align}
N_1' &\coloneqq (X_\alpha\cdot \dot{\sigma}_\alpha)\cdot \prod_{i=1}^{n-1} \big((X_i\dot{\sigma}_i) \cdot (X _{2n+1-i}\dot{\sigma}_{2n+1-i})\big);\\
N_0' &\coloneqq (u_{n+1}u_n\cdot X_\alpha\cdot \dot{\sigma}_\alpha \cdot T_\alpha)\cdot \prod_{i=1}^{n-1}\big((X_i\dot{\sigma}_i) \cdot (X _{2n+1-i}\dot{\sigma}_{2n+1-i})\big).
\end{align}
\end{defnsub}

Clearly, $N_1'$ and $N_0'$ are defined over $E'$. We would like to show that $N_1'$ is defined over $F'$, and it is possible to choose $u_n$ and $u_{n+1}$ so that $N_0'$ is defined over $F'$.

The following proposition can be verified without difficulty:
\begin{propsub}[\emph{cf.} {\cite[Lemmas~9.13,~9.14]{Steinberg:RegularElements}}] \label{prop:Steinberg-unitary}
Set $G'\coloneqq \SU_{2n+1}/ F' $ as above.
\begin{subenv}
	\item (\emph{Cf.} \cite[Lemma~9.13]{Steinberg:RegularElements}) The natural maps $X_\alpha\times \prod_{i\ne n,n+1} X_i \to N_1'$ and $X_\alpha\times T_\alpha\times \prod_{i\ne n,n+1} X_i \to N_1' $ are isomorphisms of schemes over $E'$. Furthermore,  both $N_0'$ and $N_1'$ are closed subvarieties. 
	\item\label{prop:Steinberg-unitary:indep} (\emph{Cf.} \cite[Lemma~9.14]{Steinberg:RegularElements}) The closed subvarieties $N_1'$ and $N_0'$ are independent of choices of $\dot\sigma_i$, $\dot\sigma_\alpha$ and $u_{n+1}u_{n}$ up to conjugate.
	\item (\emph{Cf.} \cite[Corollary~9.19]{Steinberg:RegularElements}) the natural map $\underline\chi$ for $G'$ defined in \eqref{eq:chi-red} induces the following isomorphisms over $F'$ 
\[N'_{1}\xrightarrow{\sim} V_1\quad   \eqref{eq:chi1} \quad \& \quad N'_{0}\xrightarrow{\sim} V_0 \quad \eqref{eq:chi0}.\]
\end{subenv}
\end{propsub}

\begin{proof}
By flat descent it suffices to verify the claims over the perfect closure of $F'$, which  is proved in the reference given.
\end{proof}

 We now show that under suitable choices we can arrange so that $N_1'$ and $N_0'$ are defined over $F'$ (and thus the isomorphisms \eqref{eq:chi1} and \eqref{eq:chi0} are defined $F'$, too).
	
	Firstly, recall that $G'\coloneqq\SU_{2n+1}/F'$ is the special unitary group associated to the hermitian space $\bigoplus_{i=1} ^{2n+1} E'e_i$ with hermitian form given by antidiagonal matrix $w$. The subgroup fixing $e_{n+1}$ is a quasi-split special unitary group
\begin{equation}\label{eq:even-unitary-subgroup}
H	'\coloneqq\SU_{2n}.	
\end{equation}	
 Using the choice of maximal torus and Borel subgroup coming from the ambient $G'\coloneqq\SU_{2n+1}$, we can embed the root datum for $\SU_{2n}$ into $\SU_{2n+1}$; in fact, the simple roots for $\SU_{2n}$ consists of $\alpha$ and $\alpha_i$ for $i\ne n,n+1$. Our choice of $\sigma_\alpha$ and $\dot\sigma_i$ ($i\ne n,n+1$) ensures that  $\dot\sigma_i\in H'(E')$ for any $i\ne n,n+1$, so we have 
 \[N'_{1,E'}\subset H'_{E'}.\] Clearly, $N_1'$ coincides with the image of closed subvarieties of normal forms for $H'$ with respect to the ordering of simple roots given by $(\alpha,\alpha_1,\alpha_{2n},\cdots)$. Therefore, $N_1'$ is defined over $F'$; \emph{cf.} Theorem~\ref{thm:NormalForms}. Indeed, the proof shows that the factor 
\begin{equation}\label{eq:fudge-factors}
\prod_{i=1}^{n-1} \big((X_i\dot{\sigma}_i) \cdot (X _{2n+1-i}\dot{\sigma}_{2n+1-i})\big) 
 \end{equation} 
  is also defined over $F'$. 
		
	Now let us consider $N_0'$. Since the root $\alpha$ is defined over $F'$, it follows that $T_\alpha$ and $X_\alpha$ are defined over $F'$. We also chose $\dot\sigma_\alpha$ in $G_\alpha(F')$. So it remains to choose $u_{n+1},u_n$ so that $u_{m+1}u_m X_\alpha$ is defined over $F'$, which is possible as explained in \cite[p~309]{Steinberg:RegularElements}.

The above discussion together with Proposition~\ref{prop:Steinberg-unitary} shows the following:
\begin{propsub}\label{prop:BS-unitary}
Set $G'\coloneqq \SU_{2n+1}/ F' $, $H'\subset G'$, $N_0'$ and $N_1'$ are defined as above.
\begin{enumerate}
	\item  $N_0'$ and $N_1'$ are closed subvarieties of $G'$ defined over $F'$. 
	\item For any $t\in T(F')\cap H'(F')$, there exists $t_0, t_1\in T(F')\cap H'(F')$ such that we have 
	\[t\cdot N_0' = t_0 N_0' t_0^{-1}\quad \& \quad t\cdot N_1' = t_1 N_1' t_1^{-1}.\]
\end{enumerate}
\end{propsub}
\begin{proof}
We have already showed the first claim, and the existence of $t_1$ in the second claim follows from Proposition~\ref{prop:Steinberg}\eqref{prop:Steinberg:indep} applied to $N_1'\subset H'$.

To find $t_0$, note that 
\[t\cdot N_0' = t\big( u_{n+1}u_n X_\alpha \big )t^{-1}\cdot t \left (X_\alpha \dot\sigma_\alpha T_\alpha\prod_{i=1}^{n-1}\big((X_i\dot{\sigma}_i) \cdot (X_{2n+1-i}\dot{\sigma}_{2n+1-i})\big)\right) \]
the factor $t (u_{n+1}u_n X_\alpha) t^{-1}$ is defined over $F'$, and we can write
\[t (u_{n+1}u_n X_\alpha) t^{-1} = u'_{n+1}u'_nX_\alpha,\]
where $u'_{n+1} = t u'_{n+1}t^{-1} \in X_{n+1} $ and $u'_{n} = tu'_{n}t^{-1} \in X_{n}$.

One can  check that for any $t_0'\in T(F')\cap H'(F')$ we have
\begin{multline*}
t_0' \left (X_\alpha \dot\sigma_\alpha T_\alpha\prod_{i=1}^{n-1}\big((X_i\dot{\sigma}_i) \cdot (X_{2n+1-i}\dot{\sigma}_{2n+1-i})\big)\right) t_0'^{-1} 
\\= t_0' \sigma''(t_0')^{-1}\left (X_\alpha \dot\sigma_\alpha T_\alpha\prod_{i=1}^{n-1}\big((X_i\dot{\sigma}_i) \cdot (X_{2n+1-i}\dot{\sigma}_{2n+1-i})\big)\right) ,
\end{multline*}
where $\sigma'' = \sigma_\alpha\prod_{i=1}^{n-1}(\sigma_i\sigma_{2n+1-i})$ is an element in the Weyl group for $H'$. Since the homomorphism $t_0'\mapsto t_0'\sigma''(t_0')^{-1}$ is  surjective on $T(F')\cap H'(F')$ (\emph{cf.} \cite[Lemma~7.5]{Steinberg:RegularElements}, or the proof of Proposition~\ref{prop:Steinberg}\eqref{prop:Steinberg:indep}), we may find $t_0'\in T(F')\cap H'(F')$ such that $t = t_0'\sigma''(t_0')^{-1}$.

Now we can write
\begin{align*}
t\cdot N_0' & =  t\big( u_{n+1}u_n X_\alpha \big )t^{-1}\cdot t_0' \left (X_\alpha \dot\sigma_\alpha T_\alpha\prod_{i=1}^{n-1}\big((X_i\dot{\sigma}_i) \cdot (X_{2n+1-i}\dot{\sigma}_{2n+1-i})\big)\right)t_0'^{-1}\\
& = t_0'\left (
(u'_{n+1}u'_n X_\alpha)\cdot X_\alpha \dot\sigma_\alpha T_\alpha\prod_{i=1}^{n-1}\big((X_i\dot{\sigma}_i) \cdot (X_{2n+1-i}\dot{\sigma}_{2n+1-i})\big)
\right )t_0'^{-1},
\end{align*}
where $u_{n+1}' = (t_0'^{-1}t)u_{n+1}(t_0'^{-1}t)^{-1}\in X_{n+1}$ and $u_{n}' = (t_0'^{-1}t)u_{n}(t_0'^{-1}t)^{-1}\in X_{n}$. Clearly, $u_{n+1}'u_{n}'X_\alpha$ is defined over $F'$ (as it is $F'$-rationally conjugate to $u_{n+1}u_nX_\alpha$, which is defined over $F'$). Since the construction of $N_0'$ is independent of the choice of $u_{n+1}, u_n$ up to conjugate (\emph{cf.}~Proposition~\ref{prop:Steinberg-unitary}\eqref{prop:Steinberg-unitary:indep}), we obtain $t_0\in T(F')\cap H'(F')$ such that $t\cdot N_0'=t_0 N_0' t_0^{-1}$.
\end{proof}

\begin{rmksub}
It is too much to expect that $t\cdot N_0'$ and $t\cdot N_1'$ are conjugate to $N_0'$ and $N_1'$ for any $t\in T(F')$, respectively. Indeed, $N_1'\subset H'$ but $T\not\subset H'$.
\end{rmksub}

\subsection{Proof of Theorem~\ref{thm1}: General case}
Since $G^{\der}$ is simply connected, we may write
\begin{equation}
G^{\der} = G^{(0)}\times \prod G^{(m)}
\end{equation}
where $G^{(0)}$ has no odd special unitary factor and $G^{(m)}$ is an $F$-simple odd special unitary group for each $m$; therefore, we may write $G^{(m)}\cong \Res_{F_m/F}\SU_{2n_m+1}$ for a finite separable extension $F_m/F$ and an integer $n_m>0$.

For each $G^{(m)}$ we consider a closed subgroup $H^{(m)}\cong \Res_{F_m/F}\SU_{2n_m}$; \emph{cf.} \eqref{eq:even-unitary-subgroup}. Let $H$ be the closed subgroup of $G$ generated by the centre of $G$, $G^{(0)}$ and $H^{(m)}$ for each $m$. We set
\begin{equation}
T_H\coloneqq T\cap H\quad \& \quad T_H^{\der}\coloneqq T\cap H^{\der}.
\end{equation}

Since $H$ is also quasi-split and $H^{\der}$ is simply connected, Lemma~\ref{lem:tori}\eqref{lem:tori:ses} can be applied to $T_H$. Furthermore, since $D\cong H/H^{\der}$, any $\bar t\in D(F)$ has a lift $t\in T_H(F)$.

In Theorem~\ref{thm:NormalForms} we defined a closed $F$-subvariety of ``normal forms'' $N^{(0)}\subset G^{(0)}$. For $G^{(m)}$, we set $N^{(m)}\coloneqq \Res_{F_m/F}N_0'\sqcup \Res_{F_m/F}N_1'$, where $N_0', N_1'\subset \SU_{2n_m+1}/F_m$ are as in Proposition~\ref {prop:BS-unitary}. We now set 
\begin{equation}
N'\coloneqq N^{(0)} \times \prod_m N^{(m)}.
\end{equation}

The following proposition concludes the proof of Theorem~\ref{thm1}.
\begin{propsub}\label{prop:BS}
	Let $C\subset G(\cl{F})$ be a  regular semi-simple conjugacy class   defined over $F$, and choose $t\in T_H(F)$ so that its image in $D(F)$ coincides with the image of $C$. Then $C$ contains a unique element in $t\cdot N'(F)$. Hence, $C$ contains an element in $G(F)$.
\end{propsub}
\begin{proof}
 We denote by $\tbar \in D(F)$ and $\underline c \in (T \sslash W)(F)$ the respective images of $C$ and fix a lift $t \in T_H(F)$ of $\tbar$. Repeating the proof of Theorem~\ref{thm:Steinberg} with $N'$ in place of $N$, we obtain a unique element $g \in t \cdot N'(F)$ with $\underline\chi(g) = \underline{c}$. Hence $g \in t \cdot N'(F) \cap C$ by Proposition~\ref{prop:class-fcn}.
\end{proof}

 \def\cprime{$'$}

 \end{document}